\documentclass{amsart}

\usepackage[all,hyperref]{paper_diening}

\usepackage{tikz}
\usetikzlibrary{calc}

\usepackage{mathtools}

\providecommand{\trSpace}{\Gamma(\partial \mathcal{T})}

\providecommand{\oQ}{\overline{Q}}

\providecommand{\snorm}[1]{{\left\vert\kern-0.25ex\left\vert\kern-0.25ex\left\vert #1   \right\vert\kern-0.25ex\right\vert\kern-0.25ex\right\vert}}

\DeclareMathOperator{\erf}{erf}

\usepackage{tikz}
\usepackage{pgfplots}
\pgfplotsset{
  width=.65\linewidth,
  axis background/.style={fill=black!5!white},
  grid style={densely dotted,semithick},
  legend style={
    legend columns=1,
    legend pos=outer north east
  },
  compat=newest 
}
\pgfplotscreateplotcyclelist{MyColors}{%
    {red,mark = *,every mark/.append style={solid,scale=0.4,fill=red}},
    {teal,mark = x,every mark/.append style={solid,scale=0.4,fill=teal}},
    {blue,mark = square*,every mark/.append style={solid,scale=0.4,fill=blue}},
{dotted,red,mark = *,every mark/.append style={solid,scale=0.4,fill=red}},
    {dotted,teal,mark = x,every mark/.append style={solid,scale=0.4,fill=teal}},
    {dotted,blue,mark = square*,every mark/.append style={solid,scale=0.4,fill=blue}},
{dash dot,red,mark = *,every mark/.append style={solid,scale=0.4,fill=red}},
    {dash dot,teal,mark = x,every mark/.append style={solid,scale=0.4,fill=teal}},
    {dash dot,blue,mark = square*,every mark/.append style={solid,scale=0.4,fill=blue}},}

\pgfplotscreateplotcyclelist{MyColors1}{%
    {red,mark = *,every mark/.append style={solid,scale=0.4,fill=red}},
    {teal,mark = x,every mark/.append style={solid,scale=0.4,fill=teal}},
    {red,dotted,mark = *,every mark/.append style={solid,scale=0.2,fill=red}},
    {teal,dotted,mark = x,every mark/.append style={solid,scale=0.2,fill=teal}},}
    
    \pgfplotscreateplotcyclelist{MyColors1b}{%
    {red,mark = *,every mark/.append style={solid,scale=0.4,fill=red}},
    {teal,mark = x,every mark/.append style={solid,scale=0.4,fill=teal}},    {teal,mark = x,every mark/.append style={solid,scale=0.4,fill=teal}},
    {red,dotted,mark = *,every mark/.append style={solid,scale=0.2,fill=red}},
    {teal,dotted,mark = x,every mark/.append style={solid,scale=0.2,fill=teal}},}
    \pgfplotscreateplotcyclelist{MyColors1bc}{%
    {red,mark = *,every mark/.append style={solid,scale=0.4,fill=red}},
    {teal,mark = x,every mark/.append style={solid,scale=0.4,fill=teal}},    
    {teal,mark = x,every mark/.append style={solid,scale=0.4,fill=teal}},
        {teal,mark = x,every mark/.append style={solid,scale=0.4,fill=teal}},
    {teal,mark = x,every mark/.append style={solid,scale=0.4,fill=teal}},
    {red,dotted,mark = *,every mark/.append style={solid,scale=0.2,fill=red}},
    {teal,dotted,mark = x,every mark/.append style={solid,scale=0.2,fill=teal}},}
    
        \pgfplotscreateplotcyclelist{MyColors1c}{%
    {red,mark = *,every mark/.append style={solid,scale=0.4,fill=red}},
    {teal,mark = x,every mark/.append style={solid,scale=0.4,fill=teal}},    {teal,mark = x,every mark/.append style={solid,scale=0.4,fill=teal}},    {teal,mark = x,every mark/.append style={solid,scale=0.4,fill=teal}},
    {red,dotted,mark = *,every mark/.append style={solid,scale=0.2,fill=red}},
    {teal,dotted,mark = x,every mark/.append style={solid,scale=0.2,fill=teal}},}

\makeatletter
\@namedef{subjclassname@2020}{%
  \textup{2020} Mathematics Subject Classification}
\makeatother

\begin{document}

\author[L.\ Diening]{Lars Diening}
\author[J.\ Storn]{Johannes Storn}
\address[L. Diening and J. Storn]{Department of Mathematics, University of Bielefeld, Postfach 10 01 31, 33501 Bielefeld, Germany}
\email{lars.diening@uni-bielefeld.de}
\email{jstorn@math.uni-bielefeld.de}

\thanks{This research was supported by the DFG through the CRC 1283 ``Taming uncertainty and profiting from randomness and low regularity in analysis, stochastics and their applications"}
\subjclass[2020]{
35K20, 
65M12, 
65M15, 
65M50,
65M60, 
}
\keywords{Parabolic PDEs, space-time FEM, DPG method, adaptivity}

\title{A Space-Time DPG Method for the Heat Equation}

\begin{abstract}
This paper introduces an ultra-weak space-time DPG method for the heat equation. We prove well-posedness of the variational formulation with broken test functions and verify quasi-optimality of a practical DPG scheme. Numerical experiments visualize beneficial properties of an adaptive and parabolically scaled mesh-refinement driven by the built-in error control of the DPG method.
\end{abstract}
\maketitle

\section{Introduction}
This paper introduces an ultra-weak space-time discontinuous Petrov--Galerkin (DPG) method for the heat equation. 
More precisely, we reformulate the heat equation as a variational problem with discontinuous test functions. 
We prove the well-posedness of the variational problem and investigate a corresponding practical DPG scheme. 
The underlying mesh consists of time-space cylinders and allows for local adaptive mesh refinements in space-time. Our investigation shows quasi-optimality of the numerical approximation and leads to a reliable and efficient error estimator. We motivate the use of parabolically scaled mesh refinements (for which our mesh is well suited) and conclude our investigation with numerical experiments. These experiments indicate that our adaptive scheme is, in particular for non-smooth initial data, superior to state-of-the-art space-time methods.

Compared to classical time-marching schemes, simultaneous space-time finite element methods lead in general to significantly larger systems of equations. On the other hand, such schemes allow for massively parallel computations and quasi-optimal a priori estimates (a necessary requirement for optimally convergent adaptive schemes). Moreover, they are well-suited for adaptive mesh refinements simultaneously in space and time.
These advantages have led to a growing interest in space-time methods for parabolic problems, see for example \cite{Steinbach15,LangerMooreNeumueller16,Moore18}. A challenging difficulty is that known well-posed simultaneous space-time variational formulations of parabolic equations are not coercive. Hence, the resulting Petrov--Galerkin schemes require pairs of discrete trial and test spaces that satisfy the so-called Ladyzhenskaya--Babu\v{s}ka--Brezzi condition. Since the design of such pairs is nontrivial, Andreev suggests a minimal residual Petrov--Galerkin discretization \cite{Andreev13}, see also \cite{StevensonWesterdiep20}. An alternative approach are least-squares finite element methods (LSFEMs), suggested in \cite{FuehrerKarkulik19} and \cite{GantnerStevenson20}. 
LSFEMs are closely related to DPG schemes \cite{Storn20}. However, the DPG method has some beneficial properties: it allows for a very flexible design of variational problems (see for example \cite{FuehrerHeuerSayay20}), might reduce to the discretization of traces (which simplifies the design of finite elements, see for example \cite{FuehrerHeuerNiemi19}), and leads to instant stable adaptive schemes resulting in good rates of convergence even in pre-asymptotic regimes \cite{BroersenDahmenStevenson18,DahmenStevenson19,PetridesDemkowicz17}.
These advantages motivate the investigation of space-time DPG methods. While space-time DPG methods for hyperbolic problems \cite{ErnestiWieners19,GopalakrishnanSepulveda19} and the Schr\"odinger equation \cite{DemkowiczGopalakrishnanNagarajSepulveda17} are already established, the existing literature on DPG methods for parabolic problems is, besides the numerical study in \cite{EllisDemkowiczChanMoser14}, limited to time-stepping schemes \cite{FuehrerHeuerNiemi17,RobertsHenneking20,FuehrerHeuerKarkulik21}.
Hence, this paper provides the first analytical foundation of space-time DPG methods for parabolic problems. To simplify the presentation as much as possible, we focus on the heat equation, even so most arguments extend to more general parabolic PDEs with ideas from \cite{GantnerStevenson20}. 

This paper is organized as follows.
Section \ref{sec:prob} reformulates the heat equation as a first-order system and analyzes the resulting system with techniques from \cite{GantnerStevenson20}. 
Section \ref{sec:DesignDPG} designs an ultra-weak DPG method that bases on the first-order system from Section \ref{sec:prob}.
The resulting variational problem is analyzed in Section \ref{sec:AnalysisDPGprob}. Section \ref{sec:abstrDiscrProb} briefly recalls abstract results from \cite{CarstensenDemkowiczGopalakrishnan14,CarstensenDemkowiczGopalakrishnan16,GopalakrishnanQiu14} for the (practical) DPG method.
These abstract results rely on the existence of a Fortin operator.
Section \ref{sec:Discretization} introduces a discretization and designs a suitable Fortin operator. 
The operator norm of the Fortin operator (and thus the DPG scheme) is stable for equally scaled meshes (the length of each cell in the mesh in time direction is equivalent to the diameter of the cell in space direction) and parabolically scaled meshes (the length of each cell in time direction is equivalent to the squared diameter of the cell in space direction). Section \ref{sec:EquiVsPara} and the numerical experiments in Section \ref{sec:NumExp} show that the latter scaling is mandatory for singular solutions. In addition, the numerical experiments in Section \ref{sec:NumExp} investigate an adaptive algorithm, driven by the built-in error control of the DPG method.

\section{Heat Equation}\label{sec:prob}
Given a time-space cylinder $Q = \mathcal{I} \times \Omega$ with time-interval $\mathcal{I} = (0,T_\textup{end}) \subset \mathbb{R}$ and bounded Lipschitz domain $\Omega \subset \mathbb{R}^d$ with boundary $\partial \Omega$ as well as a right-hand side $f\in L^2(Q)$ and some initial data $u_0 \in L^2(\Omega)$, the heat equation seeks the solution $u:Q\to \mathbb{R}$ to 
\begin{align}\label{eq:secOrder}
\partial_t u - \Delta_x u  = f\text{ in }Q,\qquad
u(0,\bigcdot) = u_0 \text{ in }\Omega,\qquad
u  = 0\text{ on }\mathcal{I}\times \partial \Omega.
\end{align}
Let $\nabla_x = (\partial_{x_1},\dots,\partial_{x_d})$ denote the gradient with respect to the space directions $x_1,\dots,x_d$ and set the divergence operator $\textup{div}\, (v,\tau) \coloneqq \partial_t v + \textup{div}_x\, \tau = \partial_t v + \partial_{x_1}\tau_1 + \dots +\partial_{x_d} \tau_d$ for scalar valued functions $v$ and vector valued functions $\tau = (\tau_1,\dots,\tau_d)$. The new unknown $\sigma \coloneqq - \nabla_x u$ leads in \eqref{eq:secOrder} to the first-order system
\begin{align}\label{eq:FOS} 
\begin{aligned}
\textup{div}\, (u,\sigma) & = f&&\text{in }Q,
&\qquad\nabla_x u + \sigma & = 0&&\text{in }Q,\\
\qquad u(0,\bigcdot) & = u_0 &&\text{in }\Omega,
&\qquad u & = 0&&\text{on }\mathcal{I}\times \partial \Omega.
\end{aligned}
\end{align}
The differential operators in \eqref{eq:FOS} are understood in a weak sense and lead to the space
\begin{align*}
U \coloneqq \big\{ (v,\tau) \in L^2(Q) \times L^2(Q;\mathbb{R}^d)\mid \nabla_x v \in L^2(Q;\mathbb{R}^d)\text{ and } \textup{div}\,(v,\tau)\in L^2(Q)\big\}.
\end{align*}
The space $U$ is a Hilbert space with norm, for all $(v,\tau)\in U$, 
\begin{align*}
\lVert (v,\tau) \rVert_U^2 \coloneqq \lVert \nabla_x v\rVert_{L^2(Q)}^2 + \lVert (v,\tau) \rVert_{L^2(Q)}^2 + \lVert \textup{div}\,(v,\tau) \rVert_{L^2(Q)}^2,
\end{align*}
where $\lVert \bigcdot \rVert_{L^2(Q)}^2 \coloneqq \langle \bigcdot ,\bigcdot\rangle_{L^2(Q)} \coloneqq \int_Q \bigcdot \cdot \bigcdot\,\mathrm{d}x$ denotes the squared $L^2$ norm.
We include the boundary condition $u = 0$ on $\mathcal{I} \times \partial\Omega$ by testing with non-compactly supported functions. This leads to the subspace $U_0\subset U$. More precisely, let $C^\infty(\oQ;\mathbb{R}^{d})$ denote the space of infinitely often differentiable vector-valued functions on the closure $\oQ$ of $Q$, then a function $\bfv = (v,\tau)$ is in $U_0$ if and only if $(v,\tau) \in U$ and
\begin{align}\label{eq:defU0}
\langle \nabla_x v ,\chi \rangle_{L^2(Q)} &= - \langle v,\textup{div}_x\,\chi\rangle_{L^2(Q)}\qquad\text{for all }\chi \in C^\infty(\oQ;\mathbb{R}^d).
\end{align}
Simple arguments (hence we skip the proof) yield the following alternative characterization of $U$ and $U_0$.
\begin{lemma}[Relation to Bochner spaces]\label{lem:BochnerSpaces}
The Sobolev spaces $U$ and $U_0$ equal the Bochner spaces from \cite{GantnerStevenson20}, that is
\begin{align*}
U &= \lbrace (v,\tau) \in L^2(\mathcal{I};H^1(\Omega)) \times L^2(Q;\mathbb{R}^d)\mid \textup{div}\,(v,\tau)\in L^2(Q)\rbrace,\\
U_0 & = \lbrace (v,\tau) \in L^2(\mathcal{I};H_0^1(\Omega)) \times L^2(Q;\mathbb{R}^d)\mid \textup{div}\,(v,\tau)\in L^2(Q)\rbrace.
\end{align*}
\end{lemma}
The following lemma from \cite[Lem.\ 2]{AntonicBurazinVrdoljak13} and  \cite[Lem.\ 2.1]{GantnerStevenson20} bounds the dual norm 
$\lVert \bigcdot \rVert_{L^2(\mathcal{I};H^{-1}(\Omega))}$ of $\lVert \bigcdot \rVert_{L^2(\mathcal{I};H^{1}_0(\Omega))} = \lVert \nabla_x\, \bigcdot\rVert_{L^2(Q)}$ for the time derivative $\partial_t v$ and so leads to a further characterization of the space $U_0$.
The lemma involves the Friedrichs constant $C_\Omega< \infty$ with
\begin{align}\label{eq:Frieidrichs}
\lVert \xi \rVert^2_{L^2(\Omega)} \leq C_\Omega \lVert \nabla_x \xi \rVert^2_{L^2(\Omega)}\qquad\text{for all }\xi \in H^1_0(\Omega).
\end{align}
\begin{lemma}[Control of time derivative]\label{lem:Embedding}
All $\bfv \coloneqq (v,\tau)\in U_0$ satisfy
\begin{align*}
\lVert \partial_t v \rVert^2_{L^2(\mathcal{I};H^{-1}(\Omega))} \leq 2\, \max\lbrace 1,C_\Omega\rbrace\,  \lVert\bfv \rVert^2_U. 
\end{align*}
\end{lemma}
\begin{proof}
Let $\bfv = (v,\tau)\in U_0$. Its divergence reads $\textup{div}\, \bfv = \textup{div}\, (v,\tau) = \partial_t v + \textup{div}_x\, \tau$. This identity and the triangle inequality yield
\begin{align*}
&\lVert \partial_t v \rVert^2_{L^2(\mathcal{I};H^{-1}(\Omega))} \leq 
2\, \lVert\textup{div}\,(v,\tau) \rVert_{L^2(\mathcal{I};H^{-1}(\Omega))}^2  + 2\, 
\lVert \textup{div}_x\, \tau \rVert_{L^2(\mathcal{I};H^{-1}(\Omega))}^2 \\
&\quad \leq  2C_\Omega\lVert \textup{div}\,(v,\tau)\rVert_{L^2(Q)}^2 + 2\, \lVert \tau\rVert_{L^2(Q)}^2 \leq 2\, \max\lbrace 1,C_\Omega\rbrace \, \lVert \bfv \rVert^2_U.\qedhere
\end{align*}
\end{proof}
As a consequence of Lemma \ref{lem:Embedding}, we have $v\in L^2(\mathcal{I};H^1_0(\Omega)) \cap H^1(\mathcal{I};H^{-1}(\Omega))$ for all $(v,\tau)\in U_0$. 
Hence, the space $U_0$ equals the ansatz space in \cite{FuehrerKarkulik19}
\begin{align*}
U_0 = \lbrace (v,\tau) \in \big(L^2(\mathcal{I};H_0^1(\Omega)) \cap H^1(\mathcal{I};H^{-1}(\Omega))\big) \times L^2(Q;\mathbb{R}^d)\mid \textup{div}\,(v,\tau)\in L^2(Q)\rbrace.
\end{align*}
Moreover, the continuous embedding
\begin{align*}
L^2(\mathcal{I};H^1_0(\Omega)) \cap H^1(\mathcal{I};H^{-1}(\Omega)) \hookrightarrow C(\overline{\mathcal{I}};L^2(\Omega))
\end{align*}
(see for example \cite[p.\ 473, Thm.\ 1]{DautrayLions92}) proves the existence of well-defined traces $\gamma_0 \bfv \coloneqq \gamma_0 v \coloneqq v(0,\bigcdot)$ and $\gamma_{T_\textup{end}} \bfv \coloneqq \gamma_{T_\textup{end}} v \coloneqq v(T_\textup{end},\bigcdot)$ in $L^2(\Omega)$ for all $\bfv = (v,\tau)\in U_0$. 
Define for all $(v,\tau) \in U$ the operators
\begin{align}\label{eq:defAandA*}
\begin{aligned}
A (v,\tau) & \coloneqq \begin{pmatrix}
\partial_t & \textup{div}_x\\
\nabla_x& \textup{id}
\end{pmatrix}
\begin{pmatrix}
v\\ \tau
\end{pmatrix}  = 
\begin{pmatrix}
\textup{div}\, (v,\tau)\\
\nabla_x v + \tau
\end{pmatrix},\\
A^* (v,\tau) & \coloneqq \begin{pmatrix}
-\partial_t & -\textup{div}_x\\
-\nabla_x& \textup{id}
\end{pmatrix}
\begin{pmatrix}
v\\ \tau
\end{pmatrix}  = 
\begin{pmatrix}
-\textup{div}\, (v,\tau)\\
-\nabla_x v + \tau
\end{pmatrix}.
\end{aligned}
\end{align}
Note that the space $U$ equals the space $H(A,Q)\coloneqq \lbrace \bfv \in L^2(Q)\times L^2(Q;\mathbb{R}^d) \mid A \bfv \in L^2(Q)\times L^2(Q;\mathbb{R}^d)\rbrace$ and the norm $\lVert \bigcdot \rVert_U$ is equivalent to the graph norm $\lVert \bigcdot \rVert_{H(A,Q)} \coloneqq (\lVert \bigcdot \rVert_{L^2(Q)}^2 + \lVert A\, \bigcdot \rVert_{L^2(Q)}^2)^{1/2}$. Moreover, we have the following.
\begin{theorem}[Equivalence of norms]\label{thm:EquiOfNorms}
For all $\bfv = (v,\tau) \in U_0$ we have
\begin{align}\label{eq:equiOfnorms}
\begin{aligned}
\lVert A \bfv\rVert_{L^2(Q)}^2 +  \lVert \gamma_0\bfv \rVert_{L^2(\Omega)}^2
\eqsim \lVert \bfv \rVert_U^2
 \eqsim \lVert A^*\bfv \rVert_{L^2(Q)}^2 + \lVert \gamma_{T_\textup{end}}\bfv \rVert_{L^2(\Omega)}^2.
\end{aligned}
\end{align}
In particular, there exists a constant $c_\bfL > 0$ with 
\begin{align}\label{eq:cl}
c_\bfL\lVert \bfv \rVert^2_{L^2(Q)} \leq \lVert A \bfv \rVert_{L^2(Q)}^2+ \lVert \gamma_0\bfv \rVert_{L^2(\Omega)}^2\qquad\text{for all }\bfv \in U_0.
\end{align}
\end{theorem}
\begin{proof}
The equivalence $\lVert A \bfv\rVert_{L^2(Q)}^2 +  \lVert \gamma_0\bfv \rVert_{L^2(\Omega)}^2
\eqsim \lVert \bfv \rVert_U^2$ for all $\bfv \in U_0$ is a special case of \cite[Thm.\ 2.3]{GantnerStevenson20}. The proof exploits the equivalence of norms 
\begin{align}\label{eq:ProofEquiNormsTemp2}
\lVert \partial_t v - \Delta_x v\rVert_{L^2(\mathcal{I};H^{-1}(\Omega))} + \lVert \gamma_0 v \rVert_{L^2(\Omega)} \eqsim \lVert \nabla_x v \rVert_{L^2(Q)} + \lVert \partial_t v \rVert_{L^2(\mathcal{I};H^{-1}(\Omega))}
\end{align}
for all $v\in L^2(\mathcal{I};H^1_0(\Omega))\times H^1(\mathcal{I};H^{-1}(\Omega))$ \cite[Thm.\ 5.1]{SchwabStevenson09}.
Similar arguments lead to the equivalence $\lVert \bfv \rVert_U^2 \eqsim \lVert A^*\bfv \rVert_{L^2(Q)}^2 + \lVert \gamma_{T_\textup{end}}\bfv \rVert_{L^2(\Omega)}^2$.
\end{proof}
\begin{remark}[Dependence on $T_\textup{end}$]\label{rem:DepenceyTend}
Let $v\in L^2(\mathcal{I};H^1_0(\Omega)) \cap H^1(\mathcal{I};H^{-1}(\Omega))$, then we have for all $t\in \mathcal{I}$ the identity
\begin{align*}
\lVert v(0,\bigcdot) \rVert_{L^2(\Omega)}^2 & = \lVert v(t,\bigcdot) \rVert_{L^2(\Omega)}^2 + 2 \int_0^t \langle \partial_t v(s,\bigcdot),v(s,\bigcdot)\rangle_{H^{-1}(\Omega),H^1_0(\Omega)} \,\mathrm{d}s.
\end{align*}
Averaging the previous identity over the time interval $\mathcal{I}$ yields
\begin{align*}
\lVert v(0,\bigcdot) \rVert_{L^2(\Omega)}^2 & \leq  \left( 1 + \frac{C_\Omega}{T_\textup{end}}\right) \lVert \nabla_x v\rVert_{L^2(Q)}^2 + \lVert \partial_t v \rVert_{L^2(\mathcal{I};H^{-1}(\Omega))}^2.
\end{align*}
Using this estimate in the proof of \eqref{eq:ProofEquiNormsTemp2} (and so of Theorem \ref{thm:EquiOfNorms}) leads to a constant $C_1 \approx 1 + 1/T_\textup{end}$ (with equivalence constants depending on $\Omega$ but not $T_\textup{end}$) which satisfies for all $\bfv \in U_0$
\begin{align*}
\lVert A \bfv\rVert_{L^2(Q)}^2 + \lVert \gamma_0\bfv \rVert_{L^2(\Omega)}^2 &\leq C_1 \lVert \bfv \rVert_U^2,\\
\lVert A^*\bfv \rVert_{L^2(Q)}^2 + \lVert \gamma_{T_\textup{end}}\bfv \rVert_{L^2(\Omega)}^2 &\leq C_1 \lVert \bfv \rVert_U^2.
\end{align*}

The hidden constant in the ``$\ \gtrsim$'' estimate in \eqref{eq:ProofEquiNormsTemp2} is uniformly bounded for all $T_\textup{end}>0$, which can be seen by a simple adjustment of the proof in \cite[Thm.\ 5.1]{SchwabStevenson09}.
This observation yields the existence of $T_\textup{end}$-independent constants $0<c_\bfL$ in \eqref{eq:cl} and $C_2<\infty$ with, for all $\bfv\in U_0$,
\begin{align*}
\lVert \bfv \rVert_U^2&\leq C_2(\lVert A \bfv\rVert_{L^2(Q)}^2 +  \lVert \gamma_0\bfv \rVert_{L^2(\Omega)}^2)\\
\lVert \bfv \rVert_U^2
 &\leq  C_2(\lVert A^*\bfv \rVert_{L^2(Q)}^2 + \lVert \gamma_{T_\textup{end}}\bfv \rVert_{L^2(\Omega)}^2).
\end{align*} 
\end{remark}
The equivalence of norms shows injectivity and continuity of the linear mapping
\begin{align*}
(A\, \bigcdot, \gamma_0\, \bigcdot):U_0 \to \big(L^2(Q) \times L^2(Q;\mathbb{R}^d)\big) \times L^2(\Omega).
\end{align*}
Surjectivity of this mapping is proven in \cite[Thm.\ 2.3]{GantnerStevenson20}. 
Combining these properties shows well-posedness of the following parabolic problem. Given $\bff\in L^2(Q)\times L^2(Q;\mathbb{R}^d)$ and $u_0 \in L^2(\Omega)$, seek $\bfu \in U_0$ with
\begin{align}\label{eq:AbstrFOS}
\begin{aligned}
A\bfu = \bff\qquad\text{and}\qquad
\gamma_0 \bfu = u_0.
\end{aligned}
\end{align}
The problem in \eqref{eq:FOS} corresponds to the right-hand side $\bff = (f,0)^\top$ with $f \in L^2(Q)$. It is possible to include right-hand sides $f\in L^2(\mathcal{I};H^{-1}(\Omega))$ by defining suitable right-hand sides $\bff \in L^2(Q)\times L^2(Q;\mathbb{R}^d)$, see \cite[Prop.\ 2.5]{GantnerStevenson20} for details.
We conclude this section with proving additional properties of the space $U_0$.
\begin{lemma}[Alternative characterizations of $U_0$]\label{lem:AlternatChar}
A function $\bfv$ is in $U_0$ if and only if $\bfv \in L^2(Q)\times L^2(Q;\mathbb{R}^d)$ and there exists a function $\Theta \in L^2(Q) \times L^2(Q;\mathbb{R}^d)$ with 
\begin{align*}
\langle \Theta, \bfw \rangle_{L^2(Q)} = 
\langle \bfv,A^* \bfw\rangle_{L^2(Q)}\qquad\text{for all }\bfw\in C_c^\infty(Q) \times C^\infty(\oQ;\mathbb{R}^d). 
\end{align*}
Alternatively, a function $\bfv$ is in $U_0$ if and only if $\bfv \in L^2(Q)\times L^2(Q;\mathbb{R}^d)$ and there exists a function $\Xi \in L^2(Q) \times L^2(Q;\mathbb{R}^d)$ with 
\begin{align*}
\langle \Xi, \bfw \rangle_{L^2(Q)} = 
\langle \bfv,A \bfw\rangle_{L^2(Q)}\qquad\text{for all }\bfw\in C_c^\infty(Q) \times C^\infty(\oQ;\mathbb{R}^d). 
\end{align*}
The functions satisfy $\Theta = A\bfv$ and $\Xi = A^*\bfv$.
\end{lemma}
\begin{proof}
The definition of $A$ and $A^*$ as well as the characterization of the space $U_0$ by the weak divergence $\textup{div}$ and gradient $\nabla_x$ in \eqref{eq:defU0} yield this lemma.
\end{proof}
The following result shows density of smooth functions in $U_0$. The set of smooth functions involves the space $C_D^\infty(\overline{Q}) \coloneqq \lbrace v \in C^\infty(\oQ) \mid v|_{\mathcal{I}\times \partial \Omega} = 0\rbrace$. A similar result is proven in \cite[Lem.\ 4]{AntonicBurazinVrdoljak13} with mollification techniques. Our proof relies on the definition of weak derivatives.
\begin{lemma}[Dense subspace]\label{lem:subspace}
Smooth functions are dense in $U_0$, that is, 
\begin{align*}
U_0  = \overline{C_D^\infty(\overline{Q}) \times C^\infty(\oQ;\mathbb{R}^d)}^{\lVert \bigcdot \rVert_U}.
\end{align*}
\end{lemma}
\begin{proof}
This proof verifies the following alternative characterization of dense subspaces: A subspace $\mathcal{U}_0$ of $U_0$ is dense if and only if every element $\zeta$ in the dual $U_0^*$ of $U_0$ that vanishes on $\mathcal{U}_0$ also vanishes on $U_0$. 

Let $\zeta \in U_0^*$ with $\zeta (\bfw) = 0$ for all $\bfw\in C_D^\infty(\overline{Q}) \times C^\infty(\oQ;\mathbb{R}^d)$. Theorem \ref{thm:EquiOfNorms} and the Riesz representation theorem imply the existence of a function $\bfv \in U_0$ with $\zeta = \langle A \bfv,A \, \bigcdot\rangle_{L^2(Q)} + \langle \gamma_0 \bfv,\gamma_0 \bigcdot\rangle_{L^2(\Omega)}$. We have
\begin{align}\label{eq:ProofDenseTemp}
\langle A \bfv,A \bfw \rangle_{L^2(Q)} + \langle \gamma_0 \bfv,\gamma_0 \bfw\rangle_{L^2(\Omega)} = 0\quad\text{for all }\bfw \in C^\infty_D(\overline{Q})\times C^\infty(\oQ;\mathbb{R}^d).
\end{align}
Testing with functions in the subspace $C^\infty_c(Q)\times C^\infty(\oQ;\mathbb{R}^d)$ reveals
\begin{align*}
\langle A \bfv,A \bfw \rangle_{L^2(Q)} = 0\qquad\text{for all }\bfw \in C^\infty_c(Q)\times C^\infty(\oQ;\mathbb{R}^d).
\end{align*}
The second characterization in Lemma \ref{lem:AlternatChar} shows that $A^*A \bfv = 0$ and $\Theta \coloneqq A \bfv \in U_0$. 
Set the subspace $C^\infty_{D,0}(\overline{Q}) \coloneqq \lbrace w \in C^\infty_D(\overline{Q})\mid \gamma_0 w = 0\rbrace$.
We have
\begin{align*}
\langle \Theta , A \bfw \rangle_{L^2(Q)} = \langle \Theta , A \bfw \rangle_{L^2(Q)} - 
\langle A^*\Theta, \bfw \rangle_{L^2(Q)} = 0\quad\text{for all }\bfw\in C_{D,0}^\infty(\overline{Q}) \times \lbrace 0 \rbrace.
\end{align*}
This is equivalent to 
$
\langle \textup{div}\, \Theta ,w \rangle_{L^2(Q)} + \langle \Theta,\nabla w\rangle_{L^2(Q)} = 0$ for all $w \in  C_{D,0}^\infty(\overline{Q})$.
Thus, the normal trace of $\Theta \in U_0 \subset H(\textup{div},Q)$ on $\lbrace T_\textup{end}\rbrace \times \Omega$ must be equal to zero, that is, the trace $\gamma_{T_\textup{end}} \Theta = 0$. The combination of $\gamma_{T_\textup{end}} \Theta = 0$, $A^* \Theta = 0$, and Theorem \ref{thm:EquiOfNorms} implies $0 = \Theta =  A\bfv$. This identity leads in \eqref{eq:ProofDenseTemp} to
\begin{align*}
\langle \gamma_0 \bfv, w\rangle_{L^2(\Omega)} = 0\qquad\text{for all }w \in C^\infty_c(\Omega).
\end{align*}
Thus, the trace $\gamma_0 \bfv = 0$. Combining this identity with $A\bfv = 0$ and Theorem \ref{thm:EquiOfNorms} shows $\bfv = 0$ and concludes the proof.
\end{proof}
Define the spaces 
\begin{align}\label{eq:DefU00}
C^\infty_{D,T_\textup{end}}(\overline{Q}) \coloneqq \lbrace w \in C^\infty_D(\overline{Q})\mid \gamma_{T_\textup{end}} w = 0\rbrace\  \text{ and }\ 
U_{00} \coloneqq \lbrace \bfv \in U_0\mid \gamma_0\bfv = 0\rbrace.
\end{align}
\begin{lemma}[Alternative characterization of $U_{00}$]\label{lem:AlternatCharU00}
A function $\bfv$ is in $U_{00}$ if and only if $\bfv \in L^2(Q)\times L^2(Q;\mathbb{R}^d)$ and there exists a function $\Theta \in L^2(Q) \times L^2(Q;\mathbb{R}^d)$ with 
\begin{align}\label{eq:CharU00}
\langle\Theta, \bfw \rangle_{L^2(Q)} = 
\langle \bfv,A^* \bfw\rangle_{L^2(Q)}\qquad\text{for all }\bfw\in C_{D,T_\textup{end}}^\infty(\overline{Q}) \times C^\infty(\oQ;\mathbb{R}^d). 
\end{align}
The function $\Theta$ satisfies $\Theta = A\bfv$.
\end{lemma}
\begin{proof}
Let $\bfv = (v,\tau) \in L^2(Q)\times L^2(Q;\mathbb{R}^d)$ with \eqref{eq:CharU00}.
Since $C_c^\infty(Q)\subset C_{D,T_\textup{end}}^\infty(\overline{Q})$, Lemma \ref{lem:AlternatChar} proves that $\bfv \in U_{0}$. Moreover, \eqref{eq:CharU00} implies
\begin{align}\label{eq:ProofCarUUU}
\langle \textup{div}\,\bfv,w\rangle_{L^2(Q)} = - \langle \bfv,\nabla w\rangle_{L^2(Q)}\qquad\text{for all }w \in C_{D,T_\textup{end}}^\infty(\overline{Q}). 
\end{align}
An integration by parts shows for sufficiently smooth functions $\bfv$ the equivalence of \eqref{eq:ProofCarUUU} and (with outer unit normal vector $\nu$)
\begin{align*}
0 = \int_{\partial Q} \bfv\cdot \nu \, w\,\mathrm{d}s = -\int_{\Omega}  \gamma_0 v \,\gamma_0 w\,\mathrm{d}s\qquad\text{for all }w\in C_{D,T_\textup{end}}^\infty(\overline{Q}).
\end{align*}
In particular, the integral $\int_\Omega \gamma_0 v\, w\, \mathrm{d}x = 0$ for all $w\in C_c^\infty(\Omega)$. This proves $\gamma_0 v = 0$ for all sufficiently smooth functions $v$. Density arguments (which apply due to Lemma \ref{lem:subspace}) conclude the proof.
\end{proof}
\section{Design of the DPG Method}\label{sec:DesignDPG}
This section reformulates the problem in \eqref{eq:AbstrFOS} as a variational problem with test functions that are discontinuous across the interfaces of some given partition $\mathcal{T}$ of the time-space cylinder $Q$. The design follows  \cite{CarstensenDemkowiczGopalakrishnan16} with abstractly defined traces as in \cite{ErnGuermond06,DemkowiczGopalakrishnanNagarajSepulveda17,Storn20}. 
Recall the definition of the (formally adjoint) operator
\begin{align*}
A^* \coloneqq \begin{pmatrix}
-\partial_t & -\textup{div}_x\\
-\nabla_x& \textup{id}
\end{pmatrix}.
\end{align*}
For all $K \in \mathcal{T}$ define the domain of the operator $A^*$ as
\begin{align*}
&H(A^*,K)  \coloneqq \lbrace \bfw \in L^2(K)\times L^2(K;\mathbb{R}^d) \mid A^* \bfw \in L^2(K)\rbrace\\
&\quad\hphantom{:} = \lbrace (w,\chi) \in  L^2(K)\times L^2(K;\mathbb{R}^d) \mid \textup{div}\, (w,\chi) \in L^2(K)\text{ and } \nabla_x w \in L^2(K;\mathbb{R}^d)\rbrace.
\end{align*}
Define the space of broken functions 
\begin{align*}
H(A^*,\mathcal{T}) & \coloneqq \lbrace \bfw \in L^2(Q) \times L^2(Q;\mathbb{R}^d) \mid \bfw|_K \in H(A^*,K)\text{ for all }K\in \mathcal{T}\rbrace,\\
Y & \coloneqq H(A^*,\mathcal{T}) \times L^2(\Omega).
\end{align*}
Let $A^*_h$ denote the element-wise application of $A^*$, that is,  
\begin{align*}
(A^*_h \bfw)|_K \coloneqq A^* (\bfw|_K)\qquad\text{for all }\bfw \in H(A^*,\mathcal{T})\text{ and }K\in \mathcal{T}.
\end{align*}
For all functions $(\bfw,\xi),(\widetilde{\bfw},\tilde{\xi})$ in the Hilbert space $Y$ the inner product reads
\begin{align}\label{eq:innerProdY}
\langle \bfw,\xi;\widetilde{\bfw},\tilde{\xi}\rangle_Y \coloneqq \langle \bfw,\widetilde{\bfw}\rangle_{L^2(Q)} + \langle A_h^*\bfw,A_h^*\widetilde{\bfw}\rangle_{L^2(Q)} + \langle \xi,\tilde{\xi} \rangle_{L^2(\Omega)}.
\end{align}
The induced norm reads 
\begin{align*}
\lVert (\bfw,\xi)\rVert_Y^2 \coloneqq \lVert \bfw \rVert_{L^2(Q)}^2 + \lVert A_h^*\bfw\rVert_{L^2(Q)}^2 + \lVert \xi \rVert_{L^2(\Omega)}^2.
\end{align*}
A multiplication of \eqref{eq:AbstrFOS} by a broken test function and an integration over $Q$ and $\Omega$ lead to the variational problem: Seek $\bfu \in U_0$ such that for all $(\bfw,\xi) \in Y$
\begin{align}\label{eq:temp1}
\langle A\bfu, \bfw\rangle_{L^2(Q)} + \langle \gamma_0 \bfu ,\xi \rangle_{L^2(\Omega)} = \langle \bff , \bfw\rangle_{L^2(Q)} + \langle u_0,\xi\rangle_{L^2(\Omega)}.
\end{align}
Define for all $\bfv = (v,\tau) \in U_0$ and $(\bfw,\xi) \in Y$ with $\bfw = (w,\chi)$ the pairing 
\begin{align}\label{eq:intByParts}
\begin{aligned}
\langle \gamma_A \bfv,( \bfw,\xi) \rangle_{\partial \mathcal{T}} & \coloneqq \langle A\bfv, \bfw\rangle_{L^2(Q)} - \langle \bfv, A^*_h \bfw \rangle_{L^2(Q)} + \langle \gamma_0 \bfv ,\xi\rangle_{L^2(\Omega)}\\
& \hphantom{:} = \sum_{K\in \mathcal{T}} \big( \langle \textup{div}\,\bfv , w\rangle_{L^2(K)} + \langle \nabla_x v ,\chi\rangle_{L^2(K)} + \langle v ,\textup{div}\,\bfw\rangle_{L^2(K)}\\
&\ \quad\qquad  + \langle \tau,\nabla_x w \rangle_{L^2(K)} \big) + \langle \gamma_0 v,\xi\rangle_{L^2(\Omega)}.
\end{aligned}
\end{align}
This pairing defines the bounded (trace) operator $\gamma_A:U_0 \to Y^*$ with 
\begin{align*}
\gamma_A \bfv \coloneqq \langle \gamma_A \bfv,\bigcdot\rangle_{\partial \mathcal{T}}.
\end{align*}
Set the functional $F(\bfw,\xi) \coloneqq \langle \bff, \bfw \rangle_{L^2(Q)}  + \langle u_0,\xi\rangle_{L^2(\Omega)}$ for all $(\bfw,\xi) \in Y$.
By design the variational problem in \eqref{eq:temp1} is equivalent to
\begin{align*}
\langle \bfu, A^*_h \bfw\rangle_{L^2(Q)} + \langle \gamma_A \bfu ,(\bfw,\xi)\rangle_{\partial \mathcal{T}} = F(\bfw,\xi)  \qquad\text{for all }(\bfw,\xi) \in Y.
\end{align*}
We define the (trace) space $\trSpace \coloneqq \gamma_A U_0 \subset Y^*$ as a subspace of the dual space $Y^*$ and set the product space 
\begin{align}\label{eq:DefSpaceX}
X \coloneqq \big( L^2(Q)\times L^2(Q;\mathbb{R}^d)\big) \times \trSpace.
\end{align}
The space $X$ is a Hilbert space with the canonical product norm $(\lVert \bfv \rVert_{L^2(Q)}^2 + \lVert \bft \rVert_{Y^*}^2)^{1/2}$ for all $(\bfv,\bft)\in X$.
Set for all $(\bfv,\bft)\in X$ and $(\bfw,\xi) \in Y$ the bilinear form $b:X\times Y \to \mathbb{R}$ as
\begin{align}
b(\bfv,\bft;\bfw,\xi) \coloneqq \langle \bfv, A^*_h \bfw\rangle_{L^2(Q)} + \langle \bft ,(\bfw,\xi)\rangle_{\partial \mathcal{T}}.
\end{align}
We introduce the trace $\bfs \coloneqq \gamma_A \bfu$ as new unknown.
Then the ultra-weak broken variational formulation of \eqref{eq:AbstrFOS} reads as follows.
\begin{definition}[Variational formulation]
Given $F\in Y^*$, we seek $(\bfu,\bfs)\in X$ with 
\begin{align}\label{eq:VarProb}
b(\bfu,\bfs;\bfw,\xi) =F(\bfw,\xi) \qquad\text{for all }(\bfw,\xi)\in Y.
\end{align}
\end{definition}
We conclude this section with the proof of the equivalence of the variational problem \eqref{eq:VarProb} and the PDE in \eqref{eq:AbstrFOS}. The proof bases upon the following two lemmas. 
\begin{lemma}[Unbroken functions]\label{lem:TraceUnbroken}
We have
\begin{align*}
\langle \bft,(\bfw,0)\rangle_{\partial\mathcal{T}} = 0\qquad\text{for all }\bft \in \trSpace\text{ and }\bfw \in C_c^\infty(Q) \times C^\infty(\oQ;\mathbb{R}^d).
\end{align*}
\end{lemma}
\begin{proof}
Let $\bft \in \trSpace$ and $\bfw \in  C_c^\infty(Q) \times C^\infty(\oQ;\mathbb{R}^d)$. The definition $\trSpace \coloneqq \gamma_A U_0$ yields the existence of a function $\bfq \in U_0$ with $\gamma_A \bfq = \bft$. Lemma \ref{lem:AlternatChar} implies
\begin{align*}
\langle \bft,(\bfw,0) \rangle_{\trSpace} & = \langle A \bfq,\bfw\rangle_{L^2(Q)} - \langle \bfq, A^*_h \bfw\rangle_{L^2(Q)}\\
& = \langle A \bfq,\bfw\rangle_{L^2(Q)} - \langle \bfq, A^* \bfw\rangle_{L^2(Q)} = 0.\qedhere
\end{align*}
\end{proof}
\begin{lemma}[Identical traces]\label{lem:identTraces}
Two traces $\bft, \bfr \in \Gamma(\partial \mathcal{T})$ are identical, if
\begin{align}\label{eq:idenTTrace}
\langle \bft,(\bfw,0)\rangle_{\partial \mathcal{T}} = \langle \bfr,(\bfw,0)\rangle_{\partial\mathcal{T}}\qquad\text{for all }\bfw \in H(A^*,\mathcal{T}).
\end{align}
\end{lemma}
\begin{proof}
Recall the definition of the spaces $U_{00}$ and $C_{D,T_\textup{end}}^\infty(\overline{Q})$ in \eqref{eq:DefU00}.
Let $\bfv \in U_0$ with $\gamma_A \bfv = \bft$ and
\begin{align*}
\langle \bft,(\bfw,0)\rangle_{\partial \mathcal{T}} = 0\qquad\text{for all }\bfw \in H(A^*,\mathcal{T}).
\end{align*}
This yields for all $\bfw \in C_{D,T_\textup{end}}^\infty(\overline{Q}) \times C^\infty(\oQ;\mathbb{R}^d) \subset H(A^*,\mathcal{T})$ that
\begin{align*}
0 = \langle \bft,(\bfw,0)\rangle_{\partial \mathcal{T}} = \langle \gamma_A \bfv,(\bfw,0)\rangle_{\partial\mathcal{T}} = \langle A \bfv ,\bfw\rangle_{L^2(Q)} - \langle \bfv,A^*\bfw\rangle_{L^2(Q)}. 
\end{align*} 
Thus Lemma \ref{lem:AlternatCharU00} shows $\bfv \in U_{00}$, that is, $\gamma_0 \bfv = 0$. This proves
\begin{align*}
\langle \bft,(\bfw,\xi)\rangle_{\partial \mathcal{T}}& = \langle \bft,(0,\xi)\rangle_{\partial \mathcal{T}} = \langle \gamma_A \bfv,(0,\xi)\rangle_{\partial \mathcal{T}}\\
& = \langle \gamma_0 \bfv,\xi\rangle_{L^2(\Omega)} = 0 \qquad\qquad \text{for all }(\bfw,\xi) \in Y.
\end{align*}
Since $\bft \in Y^*$, this shows the identity $\bft = 0$. 
\end{proof}
With the previous observations we can prove the following equivalence.
\begin{theorem}[Equivalent problems]\label{thm:equiProbs}
If $\bfu\in U_0$ solves \eqref{eq:AbstrFOS}, then the pair $(\bfu,\gamma_A \bfu) \in X$ solves \eqref{eq:VarProb}. Conversely, if the pair $(\bfu,\bfs) \in X$ solves \eqref{eq:VarProb}, then $\bfs = \gamma_A\bfu$ and $\bfu \in U_0$ solves \eqref{eq:AbstrFOS}.
\end{theorem}
\begin{proof}
It follows by the design of the variational problem \eqref{eq:VarProb} that any solution $\bfu\in U_0$ to \eqref{eq:AbstrFOS} leads to a solution $(\bfu,\gamma_A \bfu)\in X$ to \eqref{eq:VarProb}. Vice versa, let $(\bfu,\bfs)\in X$ solve \eqref{eq:VarProb}. 
Lemma \ref{lem:AlternatChar} and Lemma \ref{lem:TraceUnbroken} show that testing with smooth functions $(\bfw,0) \in \big(C_c^\infty(Q)\times C^\infty(\oQ;\mathbb{R}^d)\big) \times \lbrace 0 \rbrace \subset Y$ results in $\bfu \in U_0$ with $A\bfu = \bff$.
The combination of the integration by parts formula in \eqref{eq:intByParts} and the variational problem \eqref{eq:VarProb} shows for all $(\bfw,\xi) \in Y$ that
\begin{align}\label{eq:ProofEquiVarProb}
\begin{aligned}
&\langle A \bfu,\bfw\rangle_{L^2(Q)} + \langle \gamma_0 \bfu,\xi\rangle_{L^2(\Omega)} - \langle \gamma_A \bfu ,(\bfw,\xi)\rangle_{\partial\mathcal{T}} + \langle \bfs,(\bfw,\xi)\rangle_{\partial\mathcal{T}}\\
&\hspace*{4cm} = b(\bfu,\bfs;\bfw,\xi) = \langle \bff,\bfw\rangle_{L^2(Q)} + \langle u_0,\xi\rangle_{L^2(\Omega)}.
\end{aligned}
\end{align}
Using the identity $A\bfu = \bff$ and testing with $(\bfw,0)\in Y$ allows for the application of Lemma \ref{lem:identTraces}. This lemma shows that $\gamma_A \bfu = \bfs$. Testing in \eqref{eq:ProofEquiVarProb} with functions $(0,\xi)\in Y$ shows $\gamma_0 \bfu = u_0$ and concludes the proof. 
\end{proof}
\section{Analysis of the Variational Problem}\label{sec:AnalysisDPGprob}
Since the seminal paper \cite{CarstensenDemkowiczGopalakrishnan16} and its generalization in \cite{DemkowiczGopalakrishnanNagarajSepulveda17}, the well-posedness of broken variational formulations is well understood. Our investigation follows \cite{Storn20}, which allows us to conclude sharp inf-sup and continuity constants.

Recall the definition of the product space $X$ in \eqref{eq:DefSpaceX} and the definition of the inner product $\langle \bigcdot,\bigcdot\rangle_Y$ in the Hilbert space $Y$ in \eqref{eq:innerProdY}.
The Riesz representation theorem yields the existence of the so-called trial-to-test operator $T:X \to Y$ with
\begin{align*}
\langle T(\bfv,\bft),(\bfw,\xi)\rangle_Y = b(\bfv,\bft;\bfw,\xi) \qquad\text{for all } (\bfv,\bft)\in X\text{ and }(\bfw,\xi) \in Y.
\end{align*}
The trial-to-test operator decomposes into $T(\bfv,\bft) = \big(T_1(\bfv,\bft),T_2(\bft)\big) \in Y$, where $T_2$ is independent of $\bfv$ due to the definition of the inner product $\langle \bigcdot,\bigcdot\rangle_Y$ and the bilinear form $b$.
\begin{remark}[Trial-to-test operator]
Given some discrete subspace $X_h \subset X$ and a norm $\lVert \bigcdot\rVert_X$ in $X$, the trial-to-test operator allows to compute an ``optimal'' test space $ T X_h \subset Y$ that leads to a discrete inf-sup constant 
\begin{align*}
\beta \coloneqq \inf_{x \in X\setminus \lbrace 0 \rbrace}\sup_{y\in Y\setminus \lbrace 0 \rbrace} \frac{b(x,y)}{\lVert x\rVert_X\lVert y\rVert_Y} \leq \beta_h \coloneqq  \inf_{x_h \in X_h\setminus \lbrace 0 \rbrace}\sup_{y_h\in TX_h\setminus \lbrace 0 \rbrace} \frac{b(x_h,y_h)}{\lVert x_h\rVert_X \lVert y_h\rVert_Y}.
\end{align*}
Hence, the inf-sup stability of the continuous problem yields the inf-sup stability of the discretized problem. This was an initial motivation for the development of the DPG scheme, see for example \cite{DemkowiczGopalakrishnan10,DemkowiczGopalakrishnan11}.
\end{remark}
Recall the graph norm $\lVert \bigcdot \rVert_{H(A,Q)} = (\lVert \bigcdot\rVert^2_{L^2(Q)} + \lVert A\, \bigcdot \rVert^2_{L^2(Q)})^{1/2}$ and define for all $(\bfv,\bft)\in X$ the (mesh-dependent) operator $\mathcal{E}(\bfv,\bft) \in L^2(Q)\times L^2(Q;\mathbb{R}^d)$ by
\begin{align*}
\mathcal{E}(\bfv,\bft)& \coloneqq \bfv - A^*_h T_1(\bfv,\bft).
\end{align*}
\begin{theorem}[Extension operator]\label{thm:traceExt}
For all $(\bfv,\bft) \in X$ we have
\begin{enumerate}
\item $\mathcal{E}(\bfv,\bft) \in U_0$ with $A \mathcal{E} (\bfv,\bft) = T_1(\bfv,\bft)$,
\item $\gamma_A \mathcal{E}(\bfv,\bft) = \bft$,
\item $\lVert b(\bfv,\bft;\bigcdot)\rVert^2_{Y^*} = \lVert \bfv - \mathcal{E}(\bfv,\bft) \rVert^2_{L^2(Q)} + \lVert A \mathcal{E}(\bfv,\bft)\rVert_{L^2(Q)}^2 + \lVert \gamma_0 \mathcal{E}(\bfv,\bft) \rVert_{L^2(\Omega)}^2$,\label{itm:CharY*norm}
\item $\lVert \mathcal{E}(0,\bft)\rVert_{H(A,Q)} = \min \lbrace \lVert \bfq \rVert_{H(A,Q)}\mid \bfq \in U_0$ with $\gamma_{A} \bfq = \bft\rbrace$,\label{itm:MinExtNorm}
\item $\lVert \mathcal{E}(\bfv,0)\rVert_{H(A,Q)} \leq \lVert \bfv\rVert_{L^2(Q)}$,
\item $\lVert \mathcal{E}(\bfv,\bft)\rVert_{H(A,Q)}^2 = \lVert \mathcal{E}(0,\bft)\rVert_{H(A,Q)}^2 + \lVert \mathcal{E}(\bfv,0)\rVert_{H(A,Q)}^2$.
\end{enumerate}
\end{theorem}
\begin{proof}
This result follows (with obvious modifications) as in \cite[Thm.\ 5.1.31]{StornDiss} or \cite[Thm.\ 4.1]{Storn20}.
\end{proof}
The space $X$ is a Hilbert space with norm, for all $(\mathbf{v},\mathbf{t})\in X$,  
\begin{align}\label{eq:NormX}
\lVert (\mathbf{v},\mathbf{t}) \rVert_X^2 \coloneqq \lVert \mathbf{v} \rVert_{L^2(Q)}^2 + \lVert \mathcal{E}(\mathbf{v},\mathbf{t})\rVert_{H(A,Q)}^2  + \lVert \gamma_0 \mathcal{E}(\bfv,\bft)\rVert_{L^2(\Omega)}^2.
\end{align}

\begin{remark}[Norms]\label{cor:Norms}
Theorem \ref{thm:traceExt} and Lemma \ref{lem:identTraces} show the identities
\begin{align*}
\lVert \bft \rVert^2_{\trSpace} & \coloneqq \lVert \bft \rVert^2_{Y^*} = \lVert b(0,\bft;\bigcdot)\rVert^2_{Y^*} = \lVert T(0,\bft) \rVert_Y^2\\
&\hphantom{:} = \lVert \mathcal{E}(0,\bft) \rVert_{H(A,Q)}^2 + \lVert \gamma_0 \mathcal{E}(0,\bft)\rVert_{L^2(\Omega)}^2 \\
&\hphantom{:} = \min \lbrace \lVert \bfq \rVert^2_{H(A,Q)} + \lVert \gamma_0 \bfq \rVert_{L^2(\Omega)}^2\mid \bfq \in U_0 \text{ with }\gamma_A \bfq = \bft \rbrace.
\end{align*}
Moreover, Theorem \ref{thm:traceExt} yields the equivalence of the norm $\lVert (\mathbf{v},\mathbf{t}) \rVert_X$ and the (in the literature more common) product norm
$\smash{(\lVert \mathbf{v} \rVert_{L^2(Q)}^2 + \lVert \bft \rVert_{\Gamma(\partial \mathcal{T})}^2)^{1/2}}$ for all $(\mathbf{v},\mathbf{t}) \in X$.
\end{remark}
\begin{remark}[Duality lemma]
The equality of the dual norm and the minimal extension norm for (broken) traces is a well-known tool in the analysis of DPG formulations introduced in \cite{CarstensenDemkowiczGopalakrishnan16} for traces of $H^1$, $H(\textup{curl})$, and $H(\textup{div})$ functions.
\end{remark}
Recall the constant $c_\bfL > 0$ in Theorem \ref{thm:EquiOfNorms} and define the values 
\begin{align}\label{eq:Defbeta}
\beta  \coloneqq \frac{1+2c_\bfL - \sqrt{(1 + 2 c_\bfL)^2 - 4 c_\bfL^2}}{2 c_\bfL},\quad
\lVert b \rVert  \coloneqq \frac{1+2 c_\bfL + \sqrt{(1+2 c_\bfL)^2 - 4 c_\bfL^2}}{2c_\bfL}.
\end{align}
\begin{lemma}[Inf-sup constant]\label{lem:infSup}
We have for all $(\bfv,\bft) \in X\setminus \lbrace 0 \rbrace$ that
\begin{align}\label{eq:infSup}
0 <\beta \leq \sup_{(\bfw,\xi)\in Y\setminus \lbrace 0 \rbrace} \frac{b(\bfv,\bft;\bfw,\xi)}{\lVert (\bfv ,\bft)\rVert_X \lVert (\bfw,\xi)\rVert_Y} \leq \lVert b \rVert< \infty.
\end{align}
\end{lemma}
\begin{proof}
This result is shown in \cite[Thm.\ 5.1.46]{StornDiss} and \cite[Thm.\ 5.5]{Storn20}.
\end{proof}
\begin{theorem}[Well-posedness]\label{thm:WellPosedness}
The variational problem in \eqref{eq:VarProb} is well-posed, that is, for any right-hand side $F \in Y^*$ exists a unique solution $(\bfu,\bfs)\in X$ to
\begin{align*}
b(\bfu,\bfs;\bfw,\xi) = F(\bfw,\xi)\qquad\text{for all }(\bfw,\xi)\in Y.
\end{align*}
The solution is bounded in the sense that $
\lVert (\bfu,\bfs) \rVert_X \leq \beta^{-1} \lVert F \rVert_{Y^*}$.
\end{theorem}
\begin{proof}
We proof the well-posedness of the variational problem in \eqref{eq:VarProb} with the Babu\v{s}ka--Lax--Milgram theorem \cite[Thm.\ 2.1]{Babuska71}. Since Lemma \ref{lem:infSup} yields continuity and inf-sup condition for the bilinear form $b$, it remains to verify injectivity, that is, we verifiy the uniqueness condition 
\begin{align}\label{eq:uniquness}
\lbrace (\bfw,\xi) \in Y\mid b(\bfv,\bft;\bfw,\xi) = 0\text{ for all }(\bfv,\bft)\in X\rbrace = \lbrace 0 \rbrace.
\end{align}
Let $(\bfw,\xi) \in Y$ with $b(\bfv,\bft;\bfw,\xi) = 0$ for all $(\bfv,\bft)\in X$. This yields 
\begin{align*}
b(\bfv,0;\bfw,\xi) = \langle \bfv,A^*_h\bfw\rangle_{L^2(Q)} = 0\qquad\text{for all }\bfv \in L^2(Q)\times L^2(Q;\mathbb{R}^d).
\end{align*}
In other words, we have $A_h^* \bfw = 0$. This shows for all $\bfq \in U_0$ that
\begin{align}\label{eq:ProofWellposed}
b(0,\gamma_A \bfq;\bfw,\xi) = \langle A \bfq,\bfw \rangle_{L^2(Q)} + \langle \gamma_0 \bfq,\xi \rangle_{L^2(\Omega)} = 0.
\end{align}
In particular, all $\bfq \in C_c^\infty(Q) \times C^\infty(\oQ;\mathbb{R}^d)\subset U_0$ satisfy 
$
\langle A \bfq,\bfw \rangle_{L^2(Q)}  = 0.
$
Hence, Lemma \ref{lem:AlternatChar} shows that $\bfw \in U_0$ with $A^* \bfw = 0$.  Theorem \ref{thm:EquiOfNorms} implies $\lVert \bfw \rVert_{L^2(Q)}\lesssim \lVert  A^* \bfw\rVert_{L^2(Q)} + \lVert \gamma_{T_\textup{end}} \bfw \rVert_{L^2(\Omega)} = \lVert \gamma_{T_\textup{end}} \bfw \rVert_{L^2(\Omega)}$. Arguments as in the proof of Lemma \ref{lem:AlternatCharU00} show that $\gamma_{T_\textup{end}} \bfw = 0$ and so $\bfw = 0$. The combination of $\bfw = 0$ and \eqref{eq:ProofWellposed} implies $\xi = 0$. This proves \eqref{eq:uniquness} and concludes the proof.
\end{proof}
\begin{remark}[Optimal weight]
Remark 5.9 in \cite{Storn20} suggests the weight 
$\rho = (c_\bfL + 1)/c_\bfL$ with the constant $c_{\bfL}$ from Theorem \ref{thm:EquiOfNorms} in the test norm 
\begin{align*}
\lVert (\bfw,\xi) \rVert_{Y_\rho}^2 \coloneqq \lVert \bfw \rVert_{L^2(Q)}^2 + \rho\, \lVert A_h^*\bfw \rVert_{L^2(Q)}^2 + \rho\, \lVert \xi \rVert_{L^2(\Omega)}^2\qquad\text{for all }(\bfw,\xi)\in Y.
\end{align*} 
This weight minimizes the ratio $\beta/\lVert b\rVert$ of the continuity and inf-sup constant from Lemma~\ref{lem:infSup} and so results in an improved a priori estimate \eqref{eq:apriori}. We have skipped that weight to avoid technicalities, but suggest the use of weighted norms in $Y$; see also \cite{GopalakrishnanMugaOlivares14} for the influence of weighted norms for the Helmholtz equation.
\end{remark}
\section{Abstract Discretized Problem}\label{sec:abstrDiscrProb}
This section summarizes a priori and a posteriori error estimates for the DPG method. Let $X_h \subset X$ and $Y_h \subset Y$ be subspaces and let the functional $F(\bfw,\xi) \coloneqq \langle \bff,\bfw\rangle_{L^2(Q)} + \langle u_0 ,\xi \rangle_{L^2(\Omega)}$ for all $(\bfw,\xi) \in Y$. The DPG scheme seeks the minimizer 
\begin{align}\label{eq:DiscProbAbstr}
(\bfu_h,\bfs_h) = \argmin_{(\bfv_h,\bft_h) \in X_h}\, \lVert b(\bfv_h,\bft_h;\bigcdot) - F \rVert_{Y^*_h}.
\end{align}
This minimization is equivalent to a mixed problem and to a Galerkin FEM with optimal test functions, see \cite[Thm.\ 5.1.6--5.1.8]{StornDiss} for an overview. Notice that the use of a broken test space $Y$ results in a block-diagonal Gram matrix. This matrix can be inverted by local computations, which results in an efficient numerical scheme that solves a linear system of dimension $\dim X_h$ even if $\dim Y_h \gg \dim X_h$.
The stability of the scheme depends on the existence of a uniformly bounded discrete inf-sup constant
\begin{align}\label{eq:discrInfSup}
0 < \beta_h \coloneqq \inf_{(\bfv_h,\bft_h) \in X_h\setminus \lbrace 0 \rbrace}\sup_{(\bfw_h,\xi_h)\in Y_h\setminus \lbrace 0 \rbrace} \frac{b(\bfv_h,\bft_h;\bfw_h,\xi_h)}{\lVert(\bfv_h,\bft_h) \rVert_X \lVert (\bfw_h,\xi_h) \rVert_Y}. 
\end{align}
Lower bounds for the discrete inf-sup constant $\beta_h$ can be computed either directly (see for example \cite{CarstensenGallistlHellwigWeggler14}) or by the use (see for example \cite{CarstensenDemkowiczGopalakrishnan16,GopalakrishnanQiu14}) of a bounded Fortin operator $\Pi:Y\to Y_h$ with 
\begin{align}\label{eq:FortinAbstr}
b(\bfv_h,\bft_h;(\bfw,\xi)-\Pi (\bfw,\xi)) = 0\qquad\text{for all }(\bfv_h,\bft_h) \in X_h, (\bfw,\xi) \in Y.
\end{align}
\begin{lemma}[Equivalence]\label{lem:equiCharFortin}
Recall the constants $\beta$ and $\lVert b \rVert$ from Lemma \ref{lem:infSup}. Let $\lVert \Pi \rVert$ denote the operator norm of $\Pi$. Then \eqref{eq:FortinAbstr} implies \eqref{eq:discrInfSup} with 
\begin{align}\label{eq:eauiChar1}
\beta \, \lVert  \Pi \rVert^{-1} \leq \beta_h.
\end{align}
If \eqref{eq:discrInfSup} holds, then there exists an operator $\Pi:Y\to Y_h$ with \eqref{eq:FortinAbstr} and norm 
\begin{align}\label{eq:eauiChar2}
\lVert \Pi \rVert \leq  \beta_h^{-1} \lVert b \rVert.
\end{align}
\end{lemma}
\begin{proof}
If we have \eqref{eq:FortinAbstr}, it holds that 
\begin{align*}
\beta_h&\geq \inf_{(\bfv_h,\bft_h) \in X_h\setminus \lbrace 0 \rbrace}\sup_{(\bfw,\xi)\in Y\setminus \lbrace 0 \rbrace} \frac{b(\bfv_h,\bft_h;\Pi(\bfw,\xi))}{\lVert(\bfv_h,\bft_h) \rVert_X \lVert \Pi(\bfw,\xi) \rVert_Y}\\
& \geq \lVert \Pi \rVert^{-1} \inf_{(\bfv_h,\bft_h) \in X_h\setminus \lbrace 0 \rbrace}\sup_{(\bfw,\xi)\in Y\setminus \lbrace 0 \rbrace} \frac{b(\bfv_h,\bft_h;\bfw,\xi)}{\lVert(\bfv_h,\bft_h) \rVert_X \lVert (\bfw,\xi) \rVert_Y} \geq \beta\, \lVert \Pi \rVert^{-1}.
\end{align*}
If \eqref{eq:discrInfSup} holds, 
the Babu\v{s}ka-Lax-Milgram theorem \cite[Thm.\ 2.1]{Babuska71} implies for each $(\bfw,\xi)\in Y$ the existence of a function $(\bfw_h,\xi_h)  \in Y_h$ with 
\begin{align*}
b(\bfv_h,\bft_h;\bfw_h,\xi_h) = b(\bfv_h,\bft_h;\bfw,\xi)\qquad\text{for all }(\bfv_h,\bft_h) \in X_h.
\end{align*}
The norm $\beta_h \lVert (\bfw_h,\xi_h) \rVert_Y \leq \lVert b(\bigcdot;\bfw,\xi) \rVert_{X_h^*} \leq \lVert b \rVert\, \lVert (\bfw,\xi)\rVert_Y$. Defining $\Pi (\bfw,\xi) \coloneqq (\bfw_h,\xi_h)$ shows \eqref{eq:FortinAbstr} with upper bound \eqref{eq:eauiChar2} for the operator norm.
\end{proof}
Let $(\bfu,\bfs)\in X$ denote the solution to \eqref{eq:VarProb}.
If \eqref{eq:discrInfSup} holds, we find a unique solution $(\bfu_h,\bfs_h)\in X_h$ to \eqref{eq:DiscProbAbstr} satisfying the a~priori estimate  \cite[Thm.\ 2.1]{GopalakrishnanQiu14}
\begin{align}\label{eq:apriori}
\lVert (\bfu_h,\bfs_h) - (\bfu,\bfs)\rVert_X& \leq \lVert b \rVert\, \beta_h^{-1} \min_{(\bfv_h,\bft_h)\in X_h} \lVert (\bfv_h,\bft_h) - (\bfu,\bfs)\rVert_X.
\end{align}
Moreover, there holds the a posteriori estimate \cite[Thm.\ 2.1]{CarstensenDemkowiczGopalakrishnan14} (with improved constants from \cite[Thm.\ 5.1.4]{StornDiss})
\begin{align}\label{eq:Aposteriori}
\begin{aligned}
\beta \, \lVert(\bfu_h,\bfs_h) - (\bfu,\bfs) \rVert_X& \leq \lVert \Pi\rVert\, \lVert b(\bfu_h,\bfs_h;\bigcdot) - F\rVert_{Y_h^*} + \lVert F \circ (\textup{id} - \Pi)\rVert_{Y^*}\\
& \lesssim \lVert(\bfu_h,\bfs_h) - (\bfu,\bfs) \rVert_X.
\end{aligned}
\end{align}
The computable residual $\lVert b(\bfu_h,\bfs_h;\bigcdot) - F\rVert_{Y_h^*}$ can be evaluated locally and so can be used to drive adaptive mesh refinement schemes. 
\section{Practical DPG Scheme}\label{sec:PractDPGscheme}
In this section we introduce and analyze a practical DPG scheme for the variational problem in Section \ref{sec:DesignDPG}. We assume that $\Omega$ is a polyhedral Lipschitz domain.
\subsection{Partition}
Let $\mathcal{T}$ be a partition of the time-space cylinder $\overline{Q} = \bigcup \mathcal{T}$ into non-overlapping cells $K\in \mathcal{T}$, where all $K\subset \mathbb{R}^{d+1}$ have a positive measure $|K| >  0$ and are of the form $K = K_t \times K_x$ with time interval $K_t\subset \overline{\mathcal{J}}$ and $d$-simplex $K_x \subset \overline{\Omega}$.
We assume shape regularity in spacial direction, that is, the ratio of the diameter $h_x(K) \coloneqq \textup{diam}(K_x)$ and the maximum radius of an inscribed ball $B\subset K_x$ is uniformly bounded.
Let $\mathcal{F}$ denote the set of all facets, let $\mathcal{F}_t$ denote the facets that are orthogonal to the time axis (in other words, the first component in the normal vector is equal to zero), and let $\mathcal{F}_x = \mathcal{F} \setminus \mathcal{F}_t$ denote the subset of all facets that are parallel to the time axis.  
The partition might have hanging nodes, but we assume that facets match in the sense that
\begin{align*}
f \subseteq g\text{ or }g \subseteq f\text{ for all overlapping facets } f,g \in \mathcal{F}.
\end{align*}
Let $\mathcal{F}_c\subset \mathcal{F}$ denote the subset of all coarse facets, that is,
\begin{align*}
\mathcal{F}_c \coloneqq \lbrace f\in \mathcal{F} \mid\text{ if } f \subseteq g \text{ for some }g\in \mathcal{F}, \text{ then }f = g\rbrace. 
\end{align*}
Moreover, we define the set $\mathcal{F}_c^x \coloneqq \mathcal{F}_c \cap \mathcal{F}^x$ of all coarse facets that are parallel to the time direction. 
We set $\mathcal{T}_0$ as the union of all facets $\mathcal{F}$ on the boundary $\lbrace 0 \rbrace \times  \overline{\Omega}$.
This union is a regular triangulation of $\overline{\Omega} = \bigcup \mathcal{T}_0 \subset \mathbb{R}^d$ into $d$-simplices.
\subsection{Discretization}\label{sec:Discretization}
This subsection introduces a discretization for the ansatz space $X$ and the broken test space $Y$. 
Given some domain $R\subset \mathbb{R}^n$, $n\in \mathbb{N}$, and $k\in\mathbb{N}_0$, we define the polynomial spaces
\begin{align*}
\mathbb{P}_k(R) &\coloneqq \lbrace v \in L^\infty(R)\mid v \text{ is a polynomial of total degree }k\rbrace.
\end{align*}
We define for all $K = K_t \times K_x \in \mathcal{T}$ and $k \in \mathbb{N}_0$ the polynomial (tensor) space
\begin{align*}
\mathbb{T}_k(K)& \coloneqq 
\mathbb{P}_k(K_t) \otimes \mathbb{P}_k(K_x)\\
& \coloneqq \textup{span}\lbrace v_t v_x\mid v_t \in \mathbb{P}_k(K_t)\text{ and }v_x \in \mathbb{P}_k(K_x)\rbrace.
\end{align*}
Let $\mathcal{L}$ be a partition of some domain $\overline{R} = \bigcup \mathcal{L} \subset \mathbb{R}^n$ with $n\in \mathbb{N}$, then we define for all $k\in \mathbb{N}_0$ the spaces 
\begin{align*}
\mathbb{P}_k(\mathcal{L}) & \coloneqq \lbrace v \in L^\infty(R)\mid v|_L \in \mathbb{P}_k(L)\text{ for all }L\in \mathcal{L}\rbrace,\\
\mathbb{T}_k(\mathcal{T}) & \coloneqq \lbrace w \in L^\infty(Q)\mid w|_K \in \mathbb{T}_k(K)\text{ for all }K\in \mathcal{T}\rbrace.
\end{align*}
Set the discrete test space 
\begin{align}\label{eq:Y_hspace}
Y_h \coloneqq \begin{cases} \big(\mathbb{T}_3(\mathcal{T}) \times \mathbb{T}_1(\mathcal{T};\mathbb{R}^d)\big) \times \mathbb{P}_1(\mathcal{T}_0) \subset Y&\text{for }d = 1,\\
\big(\mathbb{T}_2(\mathcal{T}) \times \mathbb{T}_1(\mathcal{T};\mathbb{R}^d)\big) \times \mathbb{P}_1(\mathcal{T}_0) \subset Y&\text{for } d\geq 2.
\end{cases}
\end{align}
Moreover, we use piece-wise constant functions to discretize the $L^2$ space
\begin{align}\label{eq:DefX_0h}
X_{0,h} \coloneqq \mathbb{T}_0(\mathcal{T})\times \mathbb{T}_0(\mathcal{T};\mathbb{R}^d) \subset L^2(Q) \times L^2(Q;\mathbb{R}^d).
\end{align}
It remains to discretize the trace space $\Gamma(\partial \mathcal{T}) \coloneqq \gamma_A U_0$. We set the spaces $H^1_D(Q) = L^2(\mathcal{I};H^1_0(\Omega)) \cap H^1(Q)$ and let $H(\textup{div}_x,Q) = L^2(\mathcal{I};H(\textup{div},\Omega))$. Then we have 
\begin{align*}
H_D^1(Q) \times H(\textup{div}_x,Q) \subset U_0.
\end{align*}
We aim for a discretization of this dense (Lemma \ref{lem:subspace}) subspace. For the first component $\gamma_A(H_D^1(Q) \times \lbrace 0 \rbrace)$ we utilize the conforming subspace 
\begin{align}\label{eq:DefQ1D}
V_h \coloneqq \mathbb{T}_1(\mathcal{T}) \cap H^1_D(Q)
\end{align}
in the sense that our discretization contains the trace space
\begin{align*}
\Gamma_h^{H^1}(\partial \mathcal{T}) \coloneqq \gamma_A (V_h \times \lbrace 0 \rbrace) \subset \gamma_A(H_D^1(Q) \times \lbrace 0 \rbrace) \subset \Gamma(\partial \mathcal{T}).
\end{align*}
An integration by parts reveals for all functions $\bfv = (0,\tau) \in H_D^1(Q) \times H(\textup{div}_x,Q)$ and $\bfw = (w,\chi) \in H(A^*,\mathcal{T})$ that
\begin{align*}
\langle \gamma_A \bfv,\bfw\rangle_{\partial \mathcal{T}}& 
=\sum_{K \in \mathcal{T}} \int_{\partial K} \tau \cdot \nu_x\, w\,\mathrm{d}s. 
\end{align*} 
Hence, it suffices to discretize the trace $\tau\cdot \nu_x$ on each coarse facet $f \in \mathcal{F}^x_c$. We set
\begin{align*}
\Gamma_h^{\textup{div}_x}(\partial \mathcal{T}) &\coloneqq
\lbrace \gamma_A (0,\tau) \mid \tau \in H(\textup{div}_x,Q)\text{ and } \tau|_f \cdot \nu_x \in \mathbb{P}_0(f)\text{ for all }f\in \mathcal{F}_c^x\rbrace\\
&\hphantom{:}\eqsim \mathbb{P}_0(\mathcal{F}_c^x).
\end{align*}
This leads to the discrete trace space 
\begin{align*}
\Gamma_h(\partial \mathcal{T})\coloneqq \Gamma_h^{H^1}(\partial \mathcal{T}) +\Gamma_h^{\textup{div}_x}(\partial \mathcal{T})\subset \Gamma(\partial \mathcal{T}).
\end{align*}
The discrete low-order ansatz space reads
\begin{align}
X_h \coloneqq X_{0,h} \times \Gamma_h(\partial \mathcal{T}) \subset X.
\end{align}
The dimension of $X_h$ equals the dimension $\dim X_{0,h} = (d+1) \# \mathcal{T}$ plus the dimension $\dim \Gamma_h(\partial \mathcal{T})  = \# \mathcal{N}(\overline{Q} \setminus \mathcal{I} \times \partial \Omega) + \# \mathcal{F}_c^x$, where $\mathcal{N}(\overline{Q} \setminus \mathcal{I} \times \partial \Omega)$ denotes the set of all nodes that are not on the lateral boundary $ \mathcal{I} \times\partial \Omega$ of $Q$. Static condensation allows to reduce the size of the resulting linear system.
\begin{remark}[Skeleton reduction/Static condensation]\label{rem:SkeletonReduction}
Since the space $X_{0,h}$ is discontinuous across the interfaces of the underlying partition $\mathcal{T}$,
one can reduce the resulting linear system of equations. The unknowns in the reduced system are the traces, that is, the dimension of the reduced linear system equals $\dim \Gamma_h (\partial\mathcal{T})$. See \cite[Sec.\ 3]{Wieners15} for further details.
\end{remark}
\subsection{Fortin Operator}\label{sec:Fortin}
In this subsection we show the existence of a uniformly bounded Fortin operator $\Pi:Y \to Y_h$ with the discrete spaces $X_h$ and $Y_h$ defined in the previous subsection. 
In particular, we prove the following theorem.
\begin{theorem}[Fortin operator $\Pi$]\label{thm:Fortin}
It exists a linear operator $\Pi:Y\to Y_h$ with
\begin{align*}
b(\bfv_h,\bft_h;(\bfw,\xi)-\Pi(\bfw,\xi)) = 0 \qquad\text{for all }(\bfv_h,\bft_h)\in X_h\text{ and }(\bfw,\xi)\in Y.
\end{align*}
The operator is uniformly bounded in the sense that there exists a constant $C < \infty$ depending solely on $Q$ and the shape regularity in spacial direction of $\mathcal{T}$ with 
\begin{align*}
\lVert \Pi \rVert \coloneqq \sup_{(\bfw,\xi)\in Y\setminus \lbrace 0 \rbrace} \frac{\lVert \Pi (\bfw,\xi)\rVert_Y }{\lVert (\bfw,\xi) \rVert_Y} \leq C \, \max_{K\in\mathcal{T}}\lbrace 1,h_t(K),h_x(K), h_t(K)/h_x(K)\rbrace.
\end{align*}
The operator decomposes into $\Pi_0:H(A^*,\mathcal{T}) \to \mathbb{T}_3(\mathcal{T}) \times \mathbb{T}_1(\mathcal{T};\mathbb{R}^d)$ and $\Pi_1:L^2(\Omega) \to \mathbb{P}_1(\mathcal{T}_0)$ in the sense that $\Pi(\bfw,\xi) = (\Pi_0\bfw,\Pi_1\xi)$ for all $(\bfw,\xi) \in Y$. 
These operators satisfy for all $K\in \mathcal{T}$, $K_0 \in \mathcal{T}_0$, and $(\bfw,\xi) \in Y$
\begin{align}\label{eq:PropModFortin}
\begin{aligned}
\langle \bfv_0, \bfw - \Pi_0 \bfw \rangle_{L^2(K)} &= 0&&\quad\text{for all } \bfv_0 \in \mathbb{T}_0(K;\mathbb{R}^{d+1}),\\
\langle \vartheta_h, \xi - \Pi_1 \xi \rangle_{L^2(K_0)} &= 0 &&\quad
\text{for all } \vartheta_h \in \mathbb{P}_1(K_0).
\end{aligned}
\end{align}
\end{theorem}
We design the Fortin operator 
locally on each $K\in \mathcal{T}$. The local test space reads   
\begin{align*}
Y_h(K) & \coloneqq
\begin{cases}
 \mathbb{T}_3(K) \times \mathbb{T}_1(K;\mathbb{R}^d)&\text{for }d=1,\\
 \mathbb{T}_2(K) \times \mathbb{T}_1(K;\mathbb{R}^d)&\text{for }d\geq 2.
\end{cases}
\end{align*}
Recall the notation $\bfy_t \coloneqq y_0$ and $\bfy_x \coloneqq (y_1,\dots,y_d)$ for all $\bfy = (y_0,\dots,y_d)\in \mathbb{R}^{d+1}$. Let $\mathcal{F}^x(K)$ denote the set off all facets $f\in\mathcal{F}$ on the boundary of $K\in \mathcal{T}$ with normal component $\nu_x \neq 0$.
\begin{lemma}[Local Fortin operator $\Pi_K$]\label{lem:locFortin}
Let $K= K_t \times K_x\in \mathcal{T}$. Then there exists an operator $\Pi_K:H(A^*,K) \to Y_{h}(K)$ such that for all $\bfw = (w,\chi) \in H(A^*,K)$
\begin{subequations}
\begin{align}
0 &= \langle \xi_h, \textup{div}\,(\bfw - \Pi_K\bfw)\rangle_{L^2(K)} &&\text{for all } \xi_h\in \textup{div}\, Y_h(K),\label{eq:LocFortinA00}\\
0 &= \langle \bfv_0,\bfw - \Pi_K \bfw \rangle_{L^2(K)} &&\text{for all } \bfv_0 \in \mathbb{T}_0(K;\mathbb{R}^{d+1}),\label{eq:LocFortinA0}\\
0 & = \langle p_0, w - (\Pi_K \bfw)_t\rangle_{L^2(K_t \times \partial K_x)}&&\text{for all }p_0 \in \mathbb{P}_0(\mathcal{F}^x(K)),\label{eq:LocFortinB}\\
0 & = \langle \partial_t v_h, w - (\Pi_K\bfw)_t\rangle_{L^2(K)}&&\text{for all }v_h \in \mathbb{T}_1(K),\label{eq:LocFortinB2}\\
0 &= \langle \nabla_x v_h, \chi - (\Pi_K \bfw)_x\rangle_{L^2(K)} &&\text{for all }v_h\in \mathbb{T}_1(K), \label{eq:LocFortinD}\\
0 & = \langle v_h, (\bfw- \Pi_K \bfw)\cdot\nu \rangle_{L^2(\partial K)} &&\text{for all }v_h\in \mathbb{T}_1(K),\label{eq:LocFortinC}\\
0 & = \langle \tau_0, \nabla_x (w - (\Pi_K \bfw)_t)\rangle_{L^2(K)}&&\text{for all }\tau_0 \in \mathbb{T}_0(K;\mathbb{R}^d),\label{eq:LocFortinE}\\
0 &= \langle \bfv_0,A^* (\bfw - \Pi_K \bfw) \rangle_{L^2(K)} &&\text{for all } \bfv_0 \in \mathbb{T}_0(K;\mathbb{R}^{d+1}).\label{eq:LocFortinA}
\end{align}
\end{subequations}
The operator $\Pi_K$ is bounded in the sense that there exists a constant $C < \infty$ depending solely on the shape regularity in spatial direction of $\mathcal{T}$ such that 
\begin{align*}
\lVert \Pi_K \rVert& \coloneqq \sup_{\bfw\in H(A^*,K)\setminus \lbrace 0 \rbrace} \frac{\lVert \Pi_K \bfw \rVert_{H(A^*,K)}}{\lVert \bfw \rVert_{H(A^*,K})}\\
&\hphantom{:} \leq C\,\max\lbrace 1,h_t(K),h_x(K),h_t(K) / h_x(K)\rbrace.
\end{align*}
\end{lemma}
Before we design the local Fortin operator, we introduce the Piola transformation and an auxiliary result.
Let $\hat{K} = \hat{K}_t \times \hat{K}_x\subset \mathbb{R}^{d+1}$ be the reference element, where $\hat{K}_t = [0,1]$ denotes the unit interval and $\hat{K}_x\subset \mathbb{R}^d$ denotes the unit simplex. Let $F = (F_0,\dots,F_d):\hat{K} \to K\in \mathcal{T}$ be the affine mapping from the reference element $\hat{K} = \hat{K}_t \times \hat{K}_x$ onto $K = K_t \times K_x$ such that $F_0:\hat{K}_t \to K_t$ and $(F_1,\dots,F_d):\hat{K}_x \to K_x$. Set $F_t \coloneqq F_0$.
The Jacobian matrix $DF \in \mathbb{R}^{(d+1)\times (d+1)}$ decomposes into the blocks $\partial_t F_t \in \mathbb{R}$ and $D_x F_x \in \mathbb{R}^{d\times d}$ in the sense that
\begin{align}\label{eq:BlockDiag}
DF = \begin{pmatrix}
\partial_t F_t & 0\\
0 & D_xF_x
\end{pmatrix}.
\end{align}
Let $J \coloneqq \lvert\det\, (DF)\rvert$ denote the determinant of $DF$.
Given $\bfr = (r,\zeta) \in H(A^*,K)$, the Piola transformation $\hat{\bfr} = (\hat{r},\hat{\zeta}) \in H(A^*,\hat{K})$ satisfies $\bfr = J^{-1} DF \hat{\bfr} \circ F^{-1}$ and splits due to the block diagonal structure \eqref{eq:BlockDiag} into 
\begin{align}\label{eq:Transformation}
r = J^{-1} \partial_t F_t\, \hat{r} \circ F^{-1}\qquad\text{and}\qquad \zeta = J^{-1} D_xF_x \hat{\zeta} \circ F^{-1}.
\end{align} 
We have 
\begin{align*}
\nabla r = J^{-1} \partial_tF_t\big((DF)^{-1} \nabla \hat{r}\big) \circ F^{-1}\qquad\text{and}\qquad \textup{div}\, \bfr = J^{-1}( \textup{div}\, \hat{\bfr})\circ F^{-1}.
\end{align*}
The shape regularity in space implies $J \eqsim h_th_x^d\eqsim \lvert K \rvert$ and the norm equivalences
\begin{align}\label{eq:NormEquivTrans}
\begin{aligned}
&\lVert r \rVert_{L^2(K)}^2  \eqsim h_t h_x^{-d}\lVert \hat{r}\rVert_{L^2(\hat{K})}^2,
\qquad \lVert \zeta \rVert^2_{L^2(K)} \eqsim h_t^{-1}h_x^{2-d} \lVert \hat{\zeta} \rVert_{L^2(\hat{K})}^2,\\
&\lVert \textup{div}\, \bfr \rVert^2_{L^2(K)} = h_t^{-1} h_x^{-d} \, \lVert \textup{div}\, \hat{\bfr} \rVert^2_{L^2(\hat{K})},\\
&\lVert \partial_{x_j} r \rVert^2_{L^2(K)} \eqsim  h_t h_x^{-d-2}\, \lVert \partial_{x_j} \hat{r} \rVert_{L^2(\hat{K})}^2\qquad\text{ for all }j=1,\dots,d.
\end{aligned}
\end{align} 
Besides the Piola transformation, the design of $\Pi_K$ utilizes the following result.
\begin{lemma}[Auxiliary result]\label{lem:AuxProb}
Let $K = K_t\times K_x\in \mathcal{T}$ with diameter in space $h_x \coloneqq \textup{diam}(K_x)$ and faces $\mathcal{F}(K_x)$.
There exists a linear mapping $\mathcal{P}:H^1(K_x)\to \mathbb{P}_k({K}_x)$ with $k=2$ for $d\geq 2$ and $k=3$ for $d=1$ such that for all $(p_0, \xi_h) \in \mathbb{P}_0(\mathcal{F}(K_x))\times \mathbb{P}_1({K}_x)$ and $w \in H^1(K_x)$
\begin{align}\label{eq:Proofw2tnew}
\int_{\partial K_x} p_0\, {\mathcal{P}} w\,\mathrm{d}s +  \int_{K_x} \xi_h \,{\mathcal{P}}w\,\mathrm{d}x = \int_{\partial K_x} p_0\, w\,\mathrm{d}s +  \int_{K_x} \xi_h \,w\,\mathrm{d}x.
\end{align}
The operator is bounded in the sense that there exits a constant $C < \infty$ depending solely on the shape regularity in space such that 
\begin{align}\label{eq:BoundP}
 \lVert \mathcal{P} w \rVert_{L^2(K_x)}^2 + h_x^2 \lVert \nabla_x \mathcal{P} w\rVert_{L^2(K_x)}^2\leq C \left( \lVert w \rVert_{L^2(K_x)}^2 + h_x^2 \lVert \nabla_x w\rVert_{L^2(K_x)}^2 \right).
\end{align}
\end{lemma}
\begin{proof}
\textit{Step 1 (Unit simplex)}.
Let $K_x = \hat{K}_x$ be the unit simplex with barycentric coordinates $\lambda_j(x) = x_j$ and $\lambda_0(x) = 1-\sum_{j=1}^d x_j$ for all $x=(x_1,\dots,x_d)$ and $j=1,\dots,d$.
Since one can easily prove the existence of an operator $\mathcal{P} = \hat{\mathcal{P}}$ by evaluating the linear problem in \eqref{eq:Proofw2tnew} for $d=1$ by hand, we focus on the case $d\geq 2$.
To prove the existence, we show that the rank of the linear system related to \eqref{eq:Proofw2tnew} equals $\dim (\mathbb{P}_0(\mathcal{F}(\hat{K}_x))\times \mathbb{P}_1(\hat{K}_x))$. In other words, we show that there exists only the trivial solution to the problem: Seek $(q_0,\vartheta_h) \in \mathbb{P}_0(\mathcal{F}(\hat{K}_x)) \times \mathbb{P}_1(\hat{K}_x)$ with 
\begin{align}\label{eq:RankKern}
0 = \int_{\partial \hat{K}_x} q_0\,r_h \,\mathrm{d}s + \int_{\hat{K}_x} \vartheta_h \,r_h\,\mathrm{d}x \qquad\text{for all }r_h\in \mathbb{P}_k(\hat{K}_x).
\end{align}
Suppose $(q_0,\vartheta_h) \in \mathbb{P}_0(\mathcal{F}(\hat{K}_x)) \times \mathbb{P}_1(\hat{K}_x)$ solves \eqref{eq:RankKern}. 
Let $\lambda = (\lambda_1,\dots,\lambda_d)$ be the vector with all barycentric coordinates except $\lambda_0$. We fix a multi index $\sigma = (\sigma_1,\dots,\sigma_d) \in \lbrace 0 ,1\rbrace^d$ with $|\sigma| = 2$ and set the function $\psi_\sigma \coloneqq \partial^\sigma (\lambda^\sigma\lambda_0^{|\sigma|}) \in \mathbb{P}_2(\hat{K}_x)$. Integration by parts reveals
\begin{align}\label{eq:IntByParts}
\int_{\hat{K}_x} \vartheta_h\, \psi_\sigma\,\mathrm{d}x = \int_{\hat{K}_x} \vartheta_h\, \partial^\sigma (\lambda^\sigma\lambda_0^{|\sigma|})\,\mathrm{d}x = \int_{\hat{K}_x} (\partial^\sigma \vartheta_h)\, \lambda^\sigma\lambda_0^{|\sigma|}\,\mathrm{d}x = 0.
\end{align}
Let $f_j \coloneqq \lbrace x \in \hat{K}_x\mid \lambda_j(x) = 0\rbrace \in \mathcal{F}(\hat{K}_x)$ for all $j=0,\dots,d$.
Since for all multi-indices $\beta \in \lbrace 0 ,1\rbrace^d$ with $\beta \leq \sigma$ (component-wise) and $|\beta| < |\sigma|$ the function $\partial^{\beta} \lambda_0^{|\sigma|}$ equals zero on $f_0$, a calculation shows
\begin{align}\label{eq:sadfsfgsr}
\begin{aligned}
  \psi_\sigma|_{f_0}
  &= \big(\partial^\sigma (\lambda^\sigma \lambda_0^{\abs{\sigma}})\big)\big|_{f_0}= \sum_{\beta \leq \sigma} \binom{\sigma}{\beta} (\partial^{\sigma-\beta} \lambda^\sigma) (\partial^\beta \lambda_0^{\abs{\sigma}})\big|_{f_0}
  \\
  &= \lambda^\sigma (\partial^\sigma \lambda_0^{\abs{\sigma}})\big|_{f_0}= \lambda^\sigma \abs{\sigma}! (-1)^{\abs{\sigma}} \big|_{f_0} = 2 \lambda^\sigma\big|_{f_0}.
\end{aligned}
\end{align}
Let $j=1,\dots,d$ and $\sigma\in \lbrace 0 ,1\rbrace^d$ with $|\sigma| = 2$ and component $\sigma_j = 1$. Set the multi index $\gamma \in \lbrace 0,1\rbrace^d$ with entries $\gamma_k = \sigma_k$ for $k\neq j$ and $\gamma_j = 0$. A calculation shows
\begin{align*}
  \psi_\sigma|_{f_j}
  &= \big(\partial^\sigma (\lambda^\sigma \lambda_0^{\abs{\sigma}})\big)\big|_{f_j}= \big(\partial^\gamma  \partial_j^{\sigma_j} (\lambda^\gamma \lambda_j^{\sigma_j} \lambda_0^{\abs{\sigma}})\big)\big|_{f_j} 
  = \big(\partial^\gamma  (\lambda^\gamma (\partial_j^{\sigma_j} \lambda_j^{\sigma_j}) \lambda_0^{\abs{\sigma}})\big)\big|_{f_j}\\
  & = \big(\partial^\gamma  (\lambda^\gamma (\sigma_j!) \lambda_0^{\abs{\sigma}})\big)\big|_{f_j} = \sigma_j! \big(\partial^\gamma  (\lambda^\gamma \lambda_0^{\abs{\gamma}+\sigma_j})\big)\big|_{f_j}.
  \end{align*}
This identity and similar arguments as in \eqref{eq:IntByParts} show the orthogonality 
\begin{align}\label{eq:ProofTemssad}
\int_{f_j} \psi_\sigma \,q_0\,\mathrm{d}s = 0\qquad\text{for all }j=1,\dots,d.
\end{align}
Combining  \eqref{eq:RankKern}, \eqref{eq:sadfsfgsr}, and \eqref{eq:ProofTemssad} results for $r_h \coloneqq \psi_\sigma \in \mathbb{P}_2(\hat{K}_x)$ in 
\begin{align*}
0 =  \int_{\partial \hat{K}_x} q_0\,r_h \,\mathrm{d}s + \int_{\hat{K}_x} \vartheta_h \,r_h\,\mathrm{d}x = 2\int_{f_0}q_0\, \lambda^\sigma\, \mathrm{d}s.
\end{align*}
Since the integral $\int_{f_0} \lambda^\sigma\,\mathrm{d}s >0$ is not zero, we have $q_0|_{f_0} = 0$. By symmetry the same arguments show $q_0|_{f_j} = 0$ for all $j=0,\dots,d$, that is, $q_0 = 0$. Applying this identity in \eqref{eq:RankKern} shows $\vartheta_h = 0$. This proves the existence of an operator $\hat{\mathcal{P}}$ with \eqref{eq:Proofw2tnew}. 
Since $\hat{\mathcal{P}}$ maps into a finite dimensional space, the linear mapping is bounded, that is, for all $\hat{w}\in H^1(\hat{K}_x)$
\begin{align*}
\lVert \hat{\mathcal{P}} \hat{w}\rVert_{L^2(K_x)}^2 + \lVert \nabla_x \hat{\mathcal{P}} \hat{w}\rVert_{L^2(K_x)}^2 \lesssim \lVert  \hat{w}\rVert_{L^2(K_x)}^2 + \lVert \nabla_x \hat{w}\rVert_{L^2(K_x)}^2.
\end{align*}
\textit{Step 2 (Transformation).}
Let $K = K_t \times K_x \in \mathcal{T}$ with affine mapping $F_x:\hat{K}_x \to K_x$ onto $K_x$. Combining the existence of the bounded operator $\hat{\mathcal{P}}:H^1(\hat{K}_x) \to \mathbb{P}_k(\hat{K}_x)$ with the affine transformation $w \mapsto \hat{w} \circ F_x^{-1}$ verifies the existence of a linear operator $\mathcal{P}:H^1(K_x) \to \mathbb{P}_k(K_x)$ with \eqref{eq:Proofw2tnew} and \eqref{eq:BoundP}.
\end{proof}
\begin{proof}[Proof of Lemma \ref{lem:locFortin}]
Let $K = K_t \times K_x \in \mathcal{T}$ and
set the mesh sizes $h_t \coloneqq h_t(K) = |K_t|$ and $h_x \coloneqq h_x(K) = \textup{diam}(K_x)$.
Given a fixed test function $\bfw = (w,\chi) \in H(A^*,K)$, we design the function $\Pi_K \bfw \in Y_h(K)$ in the following four steps.

\textit{Step 1 (Divergence correction $\bfw_1$).}
For all $\hat{\xi}_h \in \textup{div} \, Y_h(\hat{K})$ we define the right-inverse of the divergence operator $\textup{div}_h^{-1}:\textup{div}\, Y_h(\hat{K}) \to Y_h(\hat{K})$ by
\begin{align*}
\textup{div}_h^{-1}\, \hat{\xi}_h = \argmin\big\lbrace \lVert \hat{\bfr}_h \rVert_{L^2(\hat{K})} \mid \hat{\bfr}_h \in Y_h(\hat{K})\text{ and }
\textup{div}\, \hat{\bfr}_h = \hat{\xi}_h\big\rbrace.
\end{align*}
Notice that $\int_K\textup{div}_h^{-1}\, \hat{\xi}_h \,\mathrm{d}z = 0$ for all $\hat{\xi}_h \in \textup{div}\, Y_h(\hat{K})$.
Let $\hat{\bfw} \in H(A^*,\hat{K})$ denote the Piola transformation of $\bfw\in H(A^*,K)$ and 
let the operators $\mathcal{P}_{\textup{div}\, Y_h(\hat{K})}:L^2(\hat{K}) \to \textup{div}\, Y_h(\hat{K})$ and $\mathcal{P}_{\textup{div}\, Y_h(K)} :L^2(K) \to \textup{div}\, Y_h(K)$ denote the $L^2$-orthogonal projectors onto the spaces $\textup{div}\, Y_h(\hat{K})$ and $\textup{div}\, Y_h(K)$.
We set the function  
\begin{align*}
\hat{\bfw}_1 \coloneqq \textup{div}_h^{-1} \mathcal{P}_{\textup{div}\, Y_h(\hat{K})} \textup{div}\, \hat{\bfw}.
\end{align*}
The transformed function $\bfw_1 = (w_1,\chi_1)\in H(A^*,K)$ is bounded by
\begin{align}\label{eq:Bouasdaf}
\begin{aligned}
&h_t^{-1} h_x^{d}\lVert w_1 \rVert^2_{L^2(K)} + h_t h_x^{d-2} \lVert \chi_1 \rVert^2_{L^2(K)} \eqsim \lVert \hat{\bfw}_1 \rVert_{L^2(\hat{K})}^2\\
&\qquad\qquad\qquad\qquad \lesssim \lVert \textup{div}\,\hat{\bfw}\rVert_{L^2(\hat{K})}^2  \eqsim h_t h^{d}_x \lVert\textup{div}\,\bfw\rVert_{L^2(K)}^2.  
\end{aligned}
\end{align}
Moreover, it satisfies $\textup{div}\, \bfw_1 = \mathcal{P}_{\textup{div}\, Y_h(K)} \textup{div}\, \bfw$, that is
\begin{align}\label{eq:sadasdgt}
\langle \textup{div}\, (\bfw - \bfw_1), \xi_h \rangle_{L^2(K)} = 0\qquad\text{for all }\xi_h \in \textup{div}\, Y_h(K).
\end{align}
The $L^2$-orthogonal projection property \eqref{eq:sadasdgt} implies 
\begin{align}\label{eq:sadasdgt2}
\lVert \textup{div}\, \bfw_1 \rVert_{L^2(K)} \leq \lVert \textup{div}\, \bfw \rVert_{L^2(K)}.
\end{align}
The norm equivalences \eqref{eq:NormEquivTrans} yield
\begin{align}\label{eq:boundDiv3}
\begin{aligned}
&\lVert \nabla_x w_1 \rVert^2_{L^2(K)} \eqsim h_t h_x^{-d-2}  \lVert \nabla_x \hat{w}_1 \rVert^2_{L^2(\hat{K})}\\
& \quad \qquad \lesssim  h_t h_x^{-d-2} \lVert \hat{w}_1 \rVert^2_{L^2(\hat{K})} \eqsim h_x^{-2}\, \lVert w_1 \rVert^2_{L^2(K)} \lesssim h_t^2 h_x^{-2} \lVert \textup{div}\, \bfw \rVert^2_{L^2({K})}.
\end{aligned}
\end{align}
Combining \eqref{eq:Bouasdaf}--\eqref{eq:boundDiv3} shows \eqref{eq:LocFortinA00} for $\bfw_1$ and the upper bound
\begin{align}\label{eq:w0Bound}
\begin{aligned}
\lVert \bfw_1 \rVert_{H(A^*,K)} &\lesssim (1+ h_x + h_t + h_t/h_x) \lVert \textup{div}\, \bfw\rVert_{L^2(K)} \\
&\leq (1+ h_x + h_t + h_t/h_x) \lVert \bfw \rVert_{H(A^*,K)}.
\end{aligned}
\end{align}
 
\textit{Step 2 (Integral mean correction $\bfw_2$)}.
In order to satisfy \eqref{eq:LocFortinA0}, we add the integral mean $\langle \bfw \rangle_K \coloneqq |K|^{-1} \int_K \bfw\,\mathrm{d}z$ to $\bfw_1$. Since $\langle \bfw_1 \rangle_K = 0$, the function $\bfw_2 \coloneqq \bfw_1 + \langle \bfw \rangle_K$ satisfies
\begin{align}
\langle \bfv_0, \bfw - \bfw_2 \rangle_{L^2(K)} = 
\langle \bfv_0, \bfw - \langle \bfw\rangle_K \rangle_{L^2(K)} = 0\qquad\text{for all }\bfv_0 \in \mathbb{P}_0(K;\mathbb{R}^{d+1}).
\end{align}
This shows \eqref{eq:LocFortinA00}--\eqref{eq:LocFortinA0} for $\bfw_2$. The triangle inequality and \eqref{eq:w0Bound} yield boundedness
\begin{align}\label{eq:Boundw1}
\lVert \bfw_2 \rVert_{H(A^*,K)} \lesssim (1+h_x+h_t +h_t/h_x) \lVert \bfw \rVert_{H(A^*,K)}.
\end{align}

\textit{Step 3 (Averaged  $\bfw_3$).}
Recall the bounded linear operator ${\mathcal{P}}:H^1(K_x) \to \mathbb{P}_k(K_x)$ with $k=3$ for $d=1$ and $k=2$ for $d\geq 2$ in Lemma \ref{lem:AuxProb}. Let $\bfw_2 \eqqcolon (w_2,\chi_2)$, let $\bfw \eqqcolon (w,\chi)$, and set for all $s\in K_t$ the $H^1(K_x)$ function
\begin{align*}
r(s) \coloneqq w(s,\bigcdot) - w_2(s,\bigcdot) - \int_{K_x}\big(w(s,x) - w_2(s,x)\big)\, \mathrm{d}x.
\end{align*}
We define the $\mathbb{P}_0(K_t) \otimes \mathbb{P}_k(K_x) \subset \mathbb{T}_k(K)$ function 
\begin{align*}
w_3 \coloneqq \frac{1}{|K_t|} \int_{K_t} \mathcal{P}\left(r(s)\right)\mathrm{d}s.
\end{align*}
Combining the upper bound in \eqref{eq:BoundP} and \Poincare's inequality yields 
\begin{align}\label{eq:boundAvg}
h_x^{-2} \lVert w_3 \rVert_{L^2(K)}^2 + \lVert \nabla_x w_3\rVert_{L^2(K)}^2 \lesssim \lVert \nabla_x (w -w_2) \rVert^2_{L^2(K)} \lesssim \lVert \bfw\rVert^2_{H(A^*,K)}.
\end{align}
Fubini's theorem, $\int_K w-w_2\,\mathrm{d}z = 0$, and \eqref{eq:Proofw2tnew} imply for all $v_h \in \mathbb{T}_1(K)$ that
\begin{align*}
\int_{K} \partial_t v_h \,(w-w_2-w_3)\,\mathrm{d}z & = \int_{K_x} \left(\int_{K_t}  \partial_t v_h\,\left(r(s) \right) \,\mathrm{d}s 
 -|K_t|\, \partial_t v_h\, w_{3}\right)\, \mathrm{d}x\\
 & = \int_{K_x} \int_{K_t}  \partial_t v_h\,(r(s)-\mathcal{P}(r(s)))\,\mathrm{d}s\,\mathrm{d}x = 0. 
\end{align*}
Similar arguments show that 
\begin{align*}
\int_{K_t \times \partial K_x} p_0 (w-w_2-w_3)\,\mathrm{d}s = 0\qquad\text{for all }p_0 \in \mathbb{P}_0(\mathcal{F}^x(K)). 
\end{align*}
We set the function $\bfw_3 \coloneqq \bfw_2 + (w_3,0) = (w_2+w_3,\chi_2)$, which satisfies \eqref{eq:LocFortinA00}--\eqref{eq:LocFortinB2}.
The inequalities in \eqref{eq:Boundw1} and \eqref{eq:boundAvg} yield boundedness
\begin{align}\label{eq:Boundw2-}
\lVert \bfw_3 \rVert_{H(A^*,K)} \lesssim (1+h_x+h_t +h_t/h_x) \lVert \bfw \rVert_{H(A^*,K)}.
\end{align}

\textit{Step 4 (Divergence-free correction $\bfw_4$).}
Since the gradients of discrete functions $\nabla_x \mathbb{T}_1(K)$ are contained in the space $T_1(K;\mathbb{R}^d)$ in the sense that $\nabla_x \mathbb{T}_1(K) = \mathbb{P}_1(K_t) \otimes \mathbb{P}_0(K_x;\mathbb{R}^d) \subset \mathbb{T}_1(K;\mathbb{R}^d)$, there exists the $L^2$-orthogonal projection $\chi_4 = \nabla_x r_h \in \mathbb{T}_1(K;\mathbb{R}^d)$ with $r_h \in \mathbb{T}_1(K)$ such that 
\begin{align}\label{eq:PropChi2}
\int_{K} \nabla_x v_h \cdot ({\chi} - {\chi_2} - {\chi}_4 )\,\mathrm{d}z = 0\qquad\text{for all } {v}_h \in \mathbb{T}_1(K).
\end{align}
It holds that $\textup{div}_x\, \chi_4 = \textup{div}_x\, \nabla_x r_h =0$ and the norm satisfies
\begin{align*}
\lVert {\chi}_4 \rVert_{L^2({K})}\leq \lVert {\chi} - {\chi}_2 \rVert_{L^2({K})}.
\end{align*}
Combining these observations results for $\bfw_4 \coloneqq \bfw_3 + (0,\chi_4)$ in \eqref{eq:LocFortinA00}--\eqref{eq:LocFortinD} and
\begin{align}\label{eq:Boundw2}
\lVert \bfw_4 \rVert_{H(A^*,K)}\lesssim (1+h_x+h_t +h_t/h_x) \lVert \bfw \rVert_{H(A^*,K)}.
\end{align}

\textit{Step 5 (Conclusion)}.
We set the bounded \eqref{eq:Boundw2} local Fortin operator as
\begin{align}
\Pi_K \bfw \coloneqq  \bfw_4.
\end{align}
Combining an integration by parts with \eqref{eq:LocFortinA00} (where we use $\mathbb{T}_1(K)\subset \textup{div}\, Y_h(K)$), \eqref{eq:LocFortinB2}, and \eqref{eq:LocFortinD} yields \eqref{eq:LocFortinC}. 
An integration by parts and \eqref{eq:LocFortinB} result in \eqref{eq:LocFortinE}.
The properties \eqref{eq:LocFortinA00}, \eqref{eq:LocFortinB}, and \eqref{eq:LocFortinD} imply \eqref{eq:LocFortinA}. 
\end{proof}
We define the element-wise application of the local Fortin operator $\Pi_K$ by
\begin{align*}
(\Pi_\mathcal{T} \bfw)|_K \coloneqq \Pi_K( \bfw|_K)\qquad\text{for all }K\in \mathcal{T}\text{ and }\bfw \in H(A^*,\mathcal{T}).
\end{align*}
Let $\mathcal{P}_{\mathbb{P}_1(\mathcal{T}_0)}:L^2(\Omega) \to \mathbb{P}_1(\mathcal{T}_0)$ be the $L^2$-orthogonal projector onto piece-wise affine functions $\mathbb{P}_1(\mathcal{T}_0)$. 
We set for all $(\bfw,\xi) \in Y$ the global Fortin operator
\begin{align}
\Pi \bfw \coloneqq (\Pi_\mathcal{T}\bfw,\mathcal{P}_{\mathbb{P}_1(\mathcal{T}_0)} \xi ).
\end{align}
Theorem \ref{thm:Fortin} follows from Lemma \ref{lem:locFortin} and the properties of the projector $\mathcal{P}_{\mathbb{P}_1(\mathcal{T}_0)}$.
\subsection{Error Control}\label{sec:errorControl}
Theorem \ref{thm:Fortin} allows for the application of the abstract results in Section \ref{sec:abstrDiscrProb}, including quasi-optimality. More precisely, let $(\bfu,\bfs)\in X$ denote the solution to \eqref{eq:VarProb} and let $(\bfu_h,\bfs_h) \in X_h$ be the solution to the practical DPG scheme 
\begin{align*}
(\bfu_h,\bfs_h) = \argmin_{(\bfv_h,\bft_h)\in X_h} \lVert b(\bfv_h,\bft_h;\bigcdot) - F \rVert_{Y_h^*}.
\end{align*}
Then it holds with the constants $\beta$ and $\lVert b \rVert$ from \eqref{eq:Defbeta} that
\begin{align}\label{eq:quasiOptLowOrder}
\lVert (\bfu,\bfs)-(\bfu_h,\bfs_h) \rVert_X \leq \lVert b \rVert \, \lVert \Pi \rVert\, \beta^{-1} \min_{(\bfv_h,\bft_h)\in X_h}\lVert (\bfu,\bfs) - (\bfv_h,\bft_h) \lVert_X.
\end{align}
Theorem \ref{thm:equiProbs} shows that $\bfs = \gamma_A \bfu$ and so we have by Remark \ref{cor:Norms} for all $\bfv \in L^2(Q)\times L^2(Q;\mathbb{R}^d)$ and $\bfq \in U_0$
\begin{align*}
\lVert (\bfu,\bfs) - (\bfv,\gamma_A \bfq) \rVert_X \lesssim \lVert \bfu - \bfv \rVert_{L^2(Q)} + \lVert \bfu - \bfq \rVert_{U}.
\end{align*}
This allows to bound the error in the trace approximation by the error with respect to the norm in $U_0$.
To conclude rates of convergence, a common approach uses error estimates for interpolation operators like the one derived in \cite[Sec.\ 4.1.1]{FuehrerKarkulik19}.
This result requires tensor product meshes $\mathcal{T} = \mathcal{I}_h \otimes \mathcal{T}_0$, where $\mathcal{I}_h$ is a partition of the time interval $\mathcal{I}$ and the triangulation $\mathcal{T}_0$ of $\Omega$ allows for $H^1$-stability of the $L^2$-orthogonal projector $\mathcal{P}_{W_h}:L^2(\Omega)\to W_h \coloneqq \mathbb{P}_1(\mathcal{T}_0) \cap H^1_0(\Omega)$, that is, 
\begin{align}\label{eq:L2StabH1}
\lVert \nabla_x \mathcal{P}_{W_h} v \rVert_{L^2(\Omega)}\lesssim \lVert \nabla_x v \rVert_{L^2(\Omega)}\qquad\text{for all }v\in H_0^1(\Omega).
\end{align}
This estimate holds for quasi-uniform triangulations $\mathcal{T}_0$ and certain adaptively refined meshes, c.f.\ \cite{DieningStornTscherpel20} and the references therein.
\begin{lemma}[Interpolation on tensor product meshes]\label{lem:InterpolTensMesh}
Suppose $\mathcal{T} = \mathcal{I}_h \otimes \mathcal{T}_0$ is a tensor product mesh that allows for \eqref{eq:L2StabH1}. Let $h_t = \max_{K_t\in \mathcal{I}_h} |K_t|$ and $h_x = \max_{K_x \in \mathcal{T}_0} \textup{diam}(K_x)$ denote the maximal time step and cell width. There exists an operator $I_h:C(\overline{\mathcal{I}};L^2(\Omega)) \cap L^2(\mathcal{I};H^1_0(\Omega))\to V_h$ with 
\begin{align*}
\lVert \nabla_x( u - I_hu)\rVert_{L^2(Q)} & \lesssim  h_t \lVert u \rVert_{H^1(\mathcal{I};H^1_0(\Omega))} + h_x \lVert u \rVert_{L^\infty(\mathcal{I};H^2(\Omega))},\\
\lVert \partial_t ( u - I_h u)\rVert_{L^2(\mathcal{I};H^{-1}(\Omega))} & \lesssim h_t  \lVert u \rVert_{H^2(\mathcal{I};H^{-1}(\Omega))} + h_x \lVert u \rVert_{L^\infty(\mathcal{I};H^1_0(\Omega))},\\
\lVert \partial_t (u - I_hu) \rVert_{L^2(Q)} & \lesssim h_t \lVert u \rVert_{H^2(\mathcal{I};L^2(\Omega))} + h_x \lVert u \rVert_{L^\infty(\mathcal{I};H^2(\Omega))}.
\end{align*}
\end{lemma}
\begin{proof}
This result is proven in \cite[Sec.\ 4.1.1]{FuehrerKarkulik19}. 
\end{proof}
The combination of Lemma \ref{lem:InterpolTensMesh} and the quasi-optimality \eqref{eq:quasiOptLowOrder} results for sufficiently smooth solutions in rates of convergence depending on the maximal mesh size in time $h_t$ and in space direction $h_x$.

Besides the a priori estimate in \eqref{eq:quasiOptLowOrder}, the DPG scheme allows for the built-in a posteriori error estimator in \eqref{eq:Aposteriori} consisting of the locally computable residual $\lVert b(\bfu_h,\bfs_h;\bigcdot)-F \rVert_{Y_h^*}$ and the contribution $\lVert F\circ (\textup{id}-\Pi)\rVert_{Y^*}$. Given initial data $u_0 \in L^2(\Omega)$ and a right-hand side $\bff \in L^2(Q)\times L^2(Q;\mathbb{R}^d)$, let the functional
\begin{align*}
F(\bfw,\xi) \coloneqq \langle \bff,\bfw\rangle_{L^2(Q)} + \langle u_0,\xi\rangle_{L^2(\Omega)}\qquad\text{for all } (\bfw,\xi)\in Y. 
\end{align*}
Let $\mathcal{P}_{\mathbb{T}_0(\mathcal{T})}:L^2(Q;\mathbb{R}^{d+1}) \to \mathbb{T}_0(\mathcal{T};\mathbb{R}^{d+1})$ and $\mathcal{P}_{\mathbb{P}_1(\mathcal{T}_0)}:L^2(\Omega) \to \mathbb{P}_1(\mathcal{T}_0)$ denote the $L^2$-orthogonal projectors onto $\mathbb{T}_0(\mathcal{T};\mathbb{R}^{d+1})$ and $\mathbb{P}_1(\mathcal{T}_0)$.
Property \eqref{eq:PropModFortin} of the Fortin operator $\Pi$ allows to bound the latter contribution by
\begin{align}\label{eq:DataApxNaiv}
\lVert F\circ (\textup{id}-\Pi)\rVert^2_{Y^*} \leq \lVert \textup{id} - \Pi \rVert^2 \big(\lVert \bff - \mathcal{P}_{\mathbb{T}_0(\mathcal{T})} \bff\rVert_{L^2(Q)}^2 + \lVert u_0 - \mathcal{P}_{\mathbb{P}_1(\mathcal{T}_0)} u_0\rVert_{L^2(\Omega)}^2\big).
\end{align}
The remainder of this section improves this upper bound for right-hand sides $\bff = (f,0)^\top$ with $f\in L^2(Q)$. 

Let $K = K_t \times K_x \in \mathcal{T}$ with $K_t \subset \mathbb{R}$ and $K_x \subset \mathbb{R}^d$. 
We define the $L^2$-orthogonal projector $\mathcal{P}_{\mathbb{P}_0 L^2}:L^2(K) \to \mathbb{P}_0(K_t;L^2(K_x)) \coloneqq \lbrace w \in L^2(K)\mid \partial_t w = 0\rbrace$ onto $L^2$-functions that are constant in time, that is, for all $w\in L^2(K)$ and $(x,t) \in K$ 
\begin{align*}
\mathcal{P}_{\mathbb{P}_0 L^2} w(x,t) = \dashint_{K_t} w(x,s)\,\mathrm{d}s \coloneqq \frac{1}{|K_t|}\int_{K_t} w(x,s)\,\mathrm{d}s.
\end{align*}
\begin{lemma}[Properties of the time average $\mathcal{P}_{\mathbb{P}_0 L^2}$]
Let $w \in H(A^*,K)$ with $K\in\mathcal{T}$. It holds that
\begin{align}\label{eq:propP01}
\nabla_x \mathcal{P}_{\mathbb{P}_0 L^2} w = \mathcal{P}_{\mathbb{P}_0 L^2} \nabla_x w.
\end{align}
If in addition the integral $\int_K w \,\mathrm{d}z = 0$ equals zero, we have $\int_K \mathcal{P}_{\mathbb{P}_0 L^2} w  \,\mathrm{d}z = 0$ and 
\begin{align}\label{eq:propP02}
\lVert  \mathcal{P}_{\mathbb{P}_0 L^2} w \rVert_{L^2(K)} \lesssim h_x(K) \, \lVert \mathcal{P}_{\mathbb{P}_0 L^2} \nabla_x w \rVert_{L^2(K)} \leq  h_x(K)\, \lVert  \nabla_x w \rVert_{L^2(K)}.
\end{align}
\end{lemma}
\begin{proof}
Let $w\in H(A^*,K)$ be arbitrary with cell $K\in \mathcal{T}$.

\textit{Step 1 (Proof of \eqref{eq:propP01})}.
For all $\xi \in C_0^\infty(\textup{int}(K_x);\mathbb{R}^d)$ we have 
\begin{align*}
&\langle \mathcal{P}_{\mathbb{P}_0 L^2} \nabla_x w , \xi \rangle_{L^2(K_x)} = \int_{K_x} \dashint_{K_t} \nabla_x w \cdot \xi\,\mathrm{d}s \,\mathrm{d}x = \frac{1}{|K_t|} \int_K \nabla_xw\cdot \xi \,\mathrm{d}z\\
& = \frac{-1}{|K_t|} \int_K w \,\textup{div}_x \, \xi \,\mathrm{d}z = -\int_{K_x} \dashint_{K_t} w\,\mathrm{d}s\, \textup{div}_x \,\xi \,\mathrm{d}x = -\langle  \mathcal{P}_{\mathbb{P}_0 L^2} w,\textup{div}_x \, \xi \rangle_{L^2(K_x)}.
\end{align*}
As the proof of Lemma \ref{lem:BochnerSpaces} shows, this identity is equivalent to \eqref{eq:propP01}.

\textit{Step 2 (Proof of \eqref{eq:propP02})}.
Suppose $\int_K w\,\mathrm{d}z = 0$. Fubini's theorem yields
\begin{align}\label{eq:asdfs34}
\int_K\mathcal{P}_{\mathbb{P}_0 L^2} w\,\mathrm{d}z = \frac{1}{|K_t|}  \int_K \int_{K_t} w \,\mathrm{d}s\, \mathrm{d}z = \frac{1}{|K_t|}  \int_{K_t} \int_K w \,\mathrm{d}z\, \mathrm{d}s = 0.
\end{align}
Since $\mathcal{P}_{\mathbb{P}_0 L^2} w$ is constant in time, \eqref{eq:asdfs34} implies $\int_{K_x}\mathcal{P}_{\mathbb{P}_0 L^2} w\,\mathrm{d}x = 0$. Thus, Fubini's theorem, \Poincare's inequality, shape regularity in space, and \eqref{eq:propP01} yield
\begin{align*}
\lVert \mathcal{P}_{\mathbb{P}_0 L^2} w \rVert_{L^2(K)}& \lesssim  h_x(K)\, \lVert \nabla_x \mathcal{P}_{\mathbb{P}_0 L^2}  w \rVert_{L^2(K)}\leq h_x(K)\, \lVert  \nabla_x w \rVert_{L^2(K)}.\qedhere
\end{align*}
\end{proof}
Let $\mathcal{P}_{\mathbb{T}_0(K)}:L^2(K) \to \mathbb{T}_0(K)$ denote the $L^2$-orthogonal projection onto constant functions for all $K\in \mathcal{T}$.
The properties of the time average operator $\mathcal{P}_{\mathbb{P}_0L^2}$ lead to the following result which allows us to replace the term $\lVert \bff - \mathcal{P}_{\mathbb{T}_0(\mathcal{T})}\bff\rVert_{L^2(Q)}^2$ for right-hand sides $\bff = (f,0)^\top$ with $f\in L^2(Q)$ in \eqref{eq:DataApxNaiv} by 
\begin{align*}
 \sum_{K\in \mathcal{T}} \left( h_x(K)^2 \lVert f- \mathcal{P}_{\mathbb{T}_0(K)} f\rVert_{L^2(K)}^2+  \left\lVert f - \dashint_{K_t} f \,\mathrm{d}s \right\rVert_{L^2(K)}^2\right).
\end{align*}
The first addend is of higher order. The second addend is smaller than  $\lVert f - \mathcal{P}_{\mathbb{T}_0(K)}f\rVert_{L^2(K)}^2 = \lVert \bff - \mathcal{P}_{\mathbb{T}_0(\mathcal{T})}\bff\rVert_{L^2(K)}^2$ and leads to a significant improvement of the upper bound if $f$ is smooth in time but rough in space.
\begin{theorem}[Data Approximation]\label{thm:DataApx}
Suppose the right-hand side reads $\bff = (f,0)^\top$ with $f\in L^2(Q)$. Then the data approximation error is bounded by
\begin{align*}
&\lVert F \circ (\textup{id}-\Pi) \rVert_{Y^*}^2 \lesssim \lVert u_0 - \mathcal{P}_{\mathbb{P}_1(\mathcal{T}_0)} u_0 \rVert_{L^2(\Omega)}^2 +\\
&\qquad\quad \sum_{K\in \mathcal{T}} \left( h_x(K)^2 \lVert \mathcal{P}_{\mathbb{P}_0 L^2}  (f- \mathcal{P}_{\mathbb{T}_0(K)} f)\rVert_{L^2(K)}^2+  \lVert f - \mathcal{P}_{\mathbb{P}_0 L^2} f\rVert_{L^2(K)}^2\right).
\end{align*}
The hidden constant depends solely on the uniformly bounded (Theorem \ref{thm:Fortin}) operator norm $\lVert \Pi\rVert$ and the shape regularity in space.
\end{theorem}
\begin{proof}
Let $\bff = (f,0)^\top$ with $f\in L^2(Q)$, let $u_0 \in L^2(\Omega)$, and let $(\bfw,\xi)\in Y$ with $\bfw = (w,\chi)$. We have 
\begin{align}\label{eq:proofApx0}
F((\bfw,\xi) -\Pi(\bfw,\xi)) = \langle f,w-(\Pi_\mathcal{T} \bfw)_t\rangle_{L^2(Q)} + \langle u_0, \xi - \mathcal{P}_{\mathbb{P}_1(\mathcal{T}_0)} \xi\rangle_{L^2(\Omega)}.
\end{align}
Since $\mathcal{P}_{\mathbb{P}_1(\mathcal{T}_0)}$ is the $L^2$-orthogonal projector, 
the second addend is bounded by 
\begin{align}\label{eq:proofApx1}
\langle u_0, \xi - \mathcal{P}_{\mathbb{P}_1(\mathcal{T}_0)} \xi\rangle_{L^2(\Omega)} \leq \lVert u_0 - \mathcal{P}_{\mathbb{P}_1(\mathcal{T}_0)} u_0 \rVert_{L^2(\Omega)} \lVert\xi-\mathcal{P}_{\mathbb{P}_1(\mathcal{T}_0)}\xi \rVert_{L^2(\Omega)}.
\end{align}
Set $f_0 \coloneqq f - \mathcal{P}_{\mathbb{T}_0(K)} f$.
Property \eqref{eq:PropModFortin} yields 
\begin{align*}
& \langle f,w-(\Pi_\mathcal{T} \bfw)_t\rangle_{L^2(Q)} = \langle f_0 ,w-(\Pi_\mathcal{T} \bfw)_t\rangle_{L^2(Q)}\\
& = \sum_{K\in \mathcal{T}}\left( \langle \mathcal{P}_{\mathbb{P}_0 L^2} f_0,w -  (\Pi_\mathcal{T} \bfw)_t\rangle_{L^2(K)} + \langle f_0 - \mathcal{P}_{\mathbb{P}_0 L^2} f_0,w -  (\Pi_\mathcal{T} \bfw)_t\rangle_{L^2(K)}\right).
\end{align*}
Property \eqref{eq:propP02} shows for the first addend and all $K\in \mathcal{T}$ that 
\begin{align}\label{eq:proofApx2}
\begin{aligned}
&\langle \mathcal{P}_{\mathbb{P}_0 L^2} f_0,w -  (\Pi_\mathcal{T} \bfw)_t\rangle_{L^2(K)} \\
&\qquad\leq \lVert \mathcal{P}_{\mathbb{P}_0 L^2} f_0 \rVert_{L^2(K)} \lVert \mathcal{P}_{\mathbb{P}_0 L^2} (w -  (\Pi_\mathcal{T} \bfw)_t) \rVert_{L^2(K)}\\
&\qquad \lesssim h_x(K)\, \lVert \mathcal{P}_{\mathbb{P}_0 L^2)} f_0 \rVert_{L^2(K)} \lVert\nabla_x (w -  (\Pi_\mathcal{T} \bfw)_t) \rVert_{L^2(K)}.
\end{aligned}
\end{align}
The second addend satisfies for all $K\in \mathcal{T}$
\begin{align}\label{eq:proofApx3}
\begin{aligned}
&\langle f_0 - \mathcal{P}_{\mathbb{P}_0 L^2} f_0,w -  (\Pi_\mathcal{T} \bfw)_t\rangle_{L^2(K)} \\
&\qquad\leq \lVert  f - \mathcal{P}_{\mathbb{P}_0 L^2} f \rVert_{L^2(K)} \lVert w -  (\Pi_\mathcal{T} \bfw)_t \rVert_{L^2(K)}.
\end{aligned}
\end{align}
The combination of \eqref{eq:proofApx0}--\eqref{eq:proofApx3} yields 
\begin{align*}
&F((\bfw,\xi)-\Pi(\bfw,\xi)) \lesssim \lVert (\bfw,\xi) - \Pi(\bfw,\xi)\rVert_Y \Big[\lVert u_0 - \mathcal{P}_{\mathbb{P}_1(\mathcal{T}_0)} u_0  \rVert_{L^2(\Omega)}^2 + \\
& \qquad \qquad \sum_{K\in \mathcal{T}}\left(h_x (K)^2 \lVert \mathcal{P}_{\mathbb{P}_0 L^2} f_0 \rVert_{L^2(K)}^2 + \lVert  f - \mathcal{P}_{\mathbb{P}_0 L^2} f \rVert_{L^2(K)}^2\right) \Big]^{1/2}.\qedhere
\end{align*}
\end{proof}
\subsection{Equal vs.\ Parabolic Scaling}\label{sec:EquiVsPara}
The stability of the local Fortin operators in Section \ref{sec:Fortin}
depends on the shape regularity in space and the ratio $h_t(K)/h_x(K)$, which allows for differently scaled mesh refinement strategies. In particular, theses strategies should lead to partitions $\mathcal{T}$ where the cells $K\in \mathcal{T}$ are 
\begin{enumerate}
\item equally scaled in the sense that $h_t(K)\eqsim h_x(K)$,
\item parabolically scaled  in the sense that $h_t(K)\eqsim h^2_x(K)$.
\end{enumerate}
According to the interpolation estimates in Lemma \ref{lem:InterpolTensMesh}, sufficiently smooth solutions lead to the rates of convergence 
\begin{align*}
\lVert (\bfu,\bfs) - (\bfu_h,\bfs_h) \rVert_X \lesssim \max_{K\in \mathcal{T}}\big( h_t(K) + h_x(K)\big).
\end{align*}
Let $\textup{ndof}\coloneqq \dim X_h$, then uniform mesh refinements  result in the rates
\begin{align*}
	\lVert (\bfu,\bfs) - (\bfu_h,\bfs_h) \rVert_X& \lesssim \textup{ndof}^{\,-1/(d+1)}&&\text{for equally scaled refinements},\\
	\lVert (\bfu,\bfs) - (\bfu_h,\bfs_h) \rVert_X& \lesssim \textup{ndof}^{\,-1/(d+2)}&&\text{for parabolically scaled refinements}.
\end{align*} 
These inequalities show that parabolically scaled meshes result in a reduced rate of convergence for smooth solutions. On the other hand, the following parabolic \Poincare inequality motivates parabolically scaled meshes for irregular solutions.
\begin{lemma}[Parabolic \Poincare]\label{lem:ParaPoincare}
Let $K \in \mathcal{T}$ and let
$a\in L^1(K)$ with integral mean $\mathcal{P}_{\mathbb{T}_0(K)} a \coloneqq |K|^{-1}\int_K a \,\mathrm{d}z$. Let $G\in L^1(K;\mathbb{R}^d)$ such that $\partial_t a = \textup{div}_x\, G$ in the sense of distributions. Then we have
\begin{align*}
\lVert a - \mathcal{P}_{\mathbb{T}_0(K)} a \rVert^2_{L^2(K)} \lesssim h_x(K)^2 \lVert \nabla_x a \rVert_{L^2(K)}^2 + \frac{h_t(K)^2}{h_x(K)^2}|K|^{1-2/p} \lVert G\rVert_{L^p(K)}^{2}. 
\end{align*}
\end{lemma}
\begin{proof}
This result is a special case of Lemma 2.9 in \cite{DieningSchwarzacherStroffoliniVerde17}.
\end{proof}
Suppose the function $a$ in Lemma \ref{lem:ParaPoincare} equals a gradient field $\nabla_x v$ for some $v\in L^2(\mathcal{I};H^1_0(\Omega))$. Then the identity $\partial_t \nabla_x v = \nabla_x \partial_t v = \textup{div}_x(\partial_t v\, I_{d})$ with identity matrix $I_{d} \in \mathbb{R}^{d\times d}$ shows for all $K\in \mathcal{T}$
\begin{align*}
\lVert \nabla_x v - \mathcal{P}_{\mathbb{T}_0(K)} \nabla_x v \rVert^2_{L^2(K)} \lesssim h_x(K)^2 \lVert \nabla_x^2 v \rVert_{L^2(K)}^2 + \frac{h_t(K)^2}{h_x(K)^2}|K|^{1-2/p} \lVert \partial_t v\rVert_{L^p(K)}^{2}. 
\end{align*}
Notice that the time derivative of the solution $u \in L^2(\mathcal{I};H^1_0(\Omega))$ to the heat equation \eqref{eq:secOrder} satisfies $\partial_t u =  f + \Delta_x u$. This identity results in 
\begin{align}\label{eq:term1}
\begin{aligned}
&\lVert \nabla_x u - \mathcal{P}_{\mathbb{T}_0(K)} \nabla_x u \rVert^2_{L^2(K)}\\
&\qquad\qquad \lesssim h_x(K)^2 \lVert \nabla_x^2 u \rVert_{L^2(K)}^2 + \frac{h_t(K)^2}{h_x(K)^2}|K|^{1-2/p} \lVert f + \Delta_x u\rVert_{L^p(K)}^{2}. 
\end{aligned}
\end{align}
Since the term on the left-hand side of the previous inequality occurs in the error $\lVert (\bfu,\bfs) - (\bfu_h,\bfs_h) \rVert_X$, we expect bad rates of convergence for rough right-hand sides $f$ for equally scaled mesh refinements. In particular, we obtain in \eqref{eq:term1} with $p=2$ terms of equal order, if $h_t(K) = h_x(K)^2$, i.e., for parabolically scaled meshes.
\section{Numerical Experiments}\label{sec:NumExp}
This section investigates numerically the practical DPG scheme 
\eqref{eq:DiscProbAbstr} with the discretization $X_h$ and $Y_h$ introduced in Section \ref{sec:Discretization}. In all experiments the underlying domain reads $Q = \mathcal{I} \times \Omega$ with $\mathcal{I}=\Omega = (0,1)$ and the right-hand side $\bff = (f,0)^\top$ with $f\in L^2(Q)$.
We plot errors and residuals against the degrees of freedom $\textup{ndof} \coloneqq \dim \Gamma_h(\partial \mathcal{T}) = \dim V_h + \# \mathcal{F}_c^x$.
Remark \ref{rem:SkeletonReduction} justifies this definition. We use  the \textit{QrefineR} routine in \cite{FunkenSchmidt20} to compute equally scaled meshes and we use a modified version of this routine to compute parabolically scaled meshes.
Besides uniform mesh refinements, we investigate an adaptive scheme with D\"orfler marking strategy. The refinement indicators result from the built-in error control \eqref{eq:Aposteriori}.
More precisely, let $(\eta_h,\vartheta_h) \in Y_h$ denote the computable residual that solves, with exact solution $(\bfu,\bfs) \in X$ to \eqref{eq:VarProb} and discrete solution $(\bfu_h,\bfs_h) \in X_h$ to \eqref{eq:DiscProbAbstr}, the problem
\begin{align*}
\langle \eta_h , \vartheta_h; \bfw_h,\xi_h\rangle_Y = b(\bfu_h - \bfu,\bfs_h-\bfs;\bfw_h,\xi_h) \qquad\text{for all }(\bfw_h,\xi_h) \in Y_h.
\end{align*}
Notice that $\vartheta_h \in \mathbb{P}_1(\mathcal{T}_0)$ satisfies by definition  
$\langle \vartheta_h,\xi_h\rangle_{L^2(\Omega)} = \langle \gamma_0 \bfu_h - u_0,\xi_h\rangle_{L^2(\Omega)}$ for all $\xi_h \in \mathbb{P}_1(\mathcal{T}_0)$. Hence, the Pythagorean theorem implies 
\begin{align*}
\lVert u_0 - \gamma_0 \bfu_h \rVert_{L^2(K_0)}^2 = \lVert \vartheta_h \rVert_{L^2(K_0)}^2 + \lVert u_0 - \mathcal{P}_{\mathbb{P}_1(K_0)} u_0\rVert_{L^2(K_0)}^2 \qquad\text{for all }K_0\in \mathcal{T}_0. 
\end{align*}
This observation, the a posteriori error estimate in \eqref{eq:Aposteriori}, and the upper bound for the data approximation error in Theorem \ref{thm:DataApx} lead for all $K\in \mathcal{T}$ to the error indicator
\begin{align*}
\eta^2(K) &\coloneqq \lVert \eta_h \rVert_{H(A^*,K)}^2 + h_x(K)^2 \lVert \mathcal{P}_{\mathbb{P}_0 L^2} (f - \mathcal{P}_{\mathbb{T}_0(K)}f) \rVert_{L^2(K)}^2\\
&\qquad  + \lVert u_0 - \gamma_0 \bfu_h \rVert_{L^2(K \cap \lbrace 0 \rbrace \times \Omega)}^2 + \lVert f - \mathcal{P}_{\mathbb{P}_0 L^2} f \rVert^2_{L^2(K)}.
\end{align*}
These error indicators result in the upper bound
\begin{align*}
\lVert (\bfu,\bfs) - (\bfu_h,\bfs_h) \rVert_X^2 \lesssim \eta^2(\mathcal{T}) \coloneqq \sum_{K\in \mathcal{T}} \eta^2(K).
\end{align*}
\subsection{Experiment 1 (Known Solution)}\label{sec:Exp1new}
In this experiment the solution reads 
\begin{align*}
u(t,x) = t^2 x(1-x)\qquad\text{for all }(t,x)\in Q = (0,1)^2.
\end{align*}
We use uniform mesh refinements resulting in equally and parabolically scaled meshes. 
As expected by the theory in Section \ref{sec:errorControl}, the  equally scaled refinement results in the rate of convergence $\mathcal{O}(\textup{ndof}^{-1})$ and the parabolically scaled refinement results in the rate of convergence $\mathcal{O}(\textup{ndof}^{-2/3})$  for the squared residuals and errors. Rather than computing the minimal trace extension norm $\lVert \bfs - \bfs_h \rVert_{\Gamma(\partial \mathcal{T})}$, Figure~\ref{fig:Exp1new} displays the error $\lVert u - \hat{u}_h\rVert_{L^2(Q)}^2$, where $\hat{u}_h \in V_h$ with $\gamma_A(\hat{u}_h,0) + \sigma_0 = \bfs_h$ for some $\sigma_0 \in \Gamma_h^{\textup{div}_x}(\partial\mathcal{T})$ (cf.\ Remark \ref{cor:Norms} and the discretization of the trace in Section \ref{sec:Discretization}).

\begin{figure}
{\centering
\begin{tikzpicture}
\begin{axis}[
clip=false,
width=.5\textwidth,
height=.45\textwidth,
xmode = log,
ymode = log,
cycle multi list={\nextlist MyColors},
scale = {1},
xlabel={ndof},
clip = true,
legend cell align=left,
legend style={legend columns=1,legend pos= outer north east,font=\fontsize{7}{5}\selectfont}
]
	\addplot table [x=ndof,y=residual_sq] {Experiments/Data_Exp1.txt};
	\addplot table [x=ndof,y=errU0] {Experiments/Data_Exp1.txt};
	\addplot table [x=ndof,y=L2errU] {Experiments/Data_Exp1.txt};
	\addplot table [x=ndof,y=L2errUhat] {Experiments/Data_Exp1.txt};		
	\addplot table [x=ndof,y=L2errorTau] {Experiments/Data_Exp1.txt};
	\addplot[dash dot,sharp plot,update limits=false] coordinates {(2e1,2e-2) (2e5,2e-6)};\label{plot:9}
\end{axis}
\end{tikzpicture}
\begin{tikzpicture}
\begin{axis}[
clip=false,
width=.5\textwidth,
height=.45\textwidth,
xmode = log,
ymode = log,
xlabel={ndof},
cycle multi list={\nextlist MyColors},
scale = {1},
clip = true,
legend cell align=left,
legend style={legend columns=1,legend pos= outer north east,font=\fontsize{7}{5}\selectfont}
]
	\addplot table [x=ndof,y=residual_sq] {Experiments/Data_Exp1_parabolic.txt}; \label{plot:1}
	\addplot table [x=ndof,y=errU0] {Experiments/Data_Exp1_parabolic.txt};\label{plot:2}
	\addplot table [x=ndof,y=L2errU] {Experiments/Data_Exp1_parabolic.txt};\label{plot:3}
	\addplot table [x=ndof,y=L2errUhat] {Experiments/Data_Exp1_parabolic.txt};		\label{plot:4}
	\addplot table [x=ndof,y=L2errorTau] {Experiments/Data_Exp1_parabolic.txt};\label{plot:5}
	\addplot[dashed,sharp plot,update limits=false] coordinates {(2e1,1e-2) (2e7,1e-6)};\label{plot:6}
\end{axis}
\end{tikzpicture}
\caption{Convergence history plots for the squared residual $\eta^2(\mathcal{T})$ (\ref{plot:1}) and the squared errors $\lVert \gamma_0(u - \hat{u}_h) \rVert^2_{L^2(\Omega)}$ (\ref{plot:2}),  $\lVert u - u_h \rVert_{L^2(Q)}^2$ (\ref{plot:3}), $\lVert u - \hat{u}_h \rVert_{L^2(Q)}^2$ (\ref{plot:4}), and $\lVert \sigma - \sigma_h \rVert_{L^2(Q)}^2$ (\ref{plot:5}) for Experiment 1 (Known Solution) with equally scaled (left) and parabolically scaled refinement (right). The dash dotted line (\ref{plot:9}) indicates the rate $\mathcal{O}(\textup{ndof}^{-1})$ and dashed line (\ref{plot:6}) indicates the rate $\mathcal{O}(\textup{ndof}^{-2/3})$.} \label{fig:Exp1new}}
\end{figure}
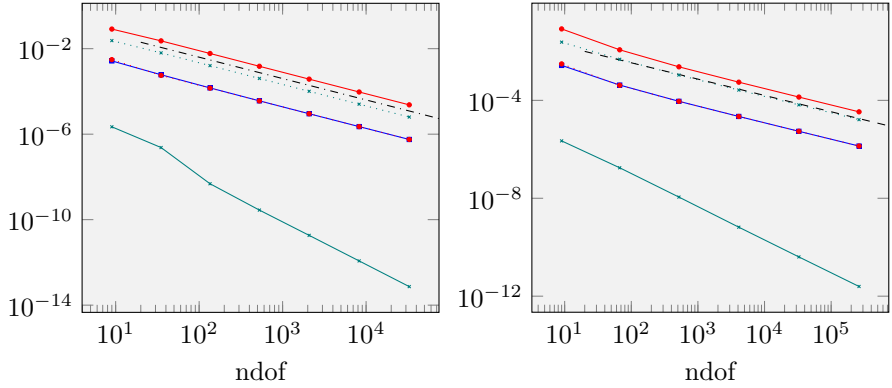
\subsection{Experiment 2 (Checkerboard)}\label{sec:Exp3new}
This experiment investigates our numerical scheme with jumping right-hand side $f:Q \to \lbrace -1,1\rbrace$ displayed in Figure \ref{fig:Exp3new} and initial data $u_0 = 0$. Since the initial data is smooth, maximal regularity \cite{Fackler17} suggests
$\partial_t u \in L^p(Q)$ for all $p<\infty$ and $\Delta_x u \in L^2(Q)$, which implies $\nabla^2_x u \in L^2(Q;\mathbb{R}^d)$.
Hence, the upper bound in \eqref{eq:term1} suggests the optimal order of convergence. Indeed, the convergence history plot in Figure \ref{fig:Exp3new} shows that the residual $\eta^2(\mathcal{T})$ converges with the optimal rate $\mathcal{O}(\textup{ndof}^{-2/3})$ for adaptive and uniform parabolically scaled refinements.
The rate of convergence with adaptive and uniform equally scaled refinements reads $\mathcal{O}(\textup{ndof}^{-1})$.
\begin{figure}
{\centering
\begin{tikzpicture}
\begin{axis}[
clip=false,
width=.5\textwidth,
height=.45\textwidth,
xmode = log,
ymode = log,
xlabel={ndof},
cycle multi list={\nextlist MyColors1c},
scale = {1},
clip = true,
legend cell align=left,
legend style={legend columns=1,legend pos= south west,font=\fontsize{7}{5}\selectfont}
]
\addplot table [x=ndof,y=residual_sq] {Experiments/Data_Exp4_uniform.txt};
\addplot table [x=ndof,y=residual_sq] {Experiments/Data_Exp4_parabolic_uniform.txt};
\addplot[dash dot,sharp plot,update limits=false] coordinates {(2e1,3.6e-2) (2e5,3.6e-6)};
\addplot[dotted,sharp plot,update limits=false] coordinates {(1e1,2e-2) (1e7,2e-6)};
\addplot table [x=ndof,y=residual_sq] {Experiments/Data_Exp4_adaptive.txt};
\addplot table [x=ndof,y=residual_sq] {Experiments/Data_Exp4_parabolic_adaptive.txt};	


	\legend{
	{equal},
	{parabolic},
	{$\mathcal{O}(\textup{ndof}^{-1}$)},
	{$\mathcal{O}(\textup{ndof}^{-2/3}$)}};
\end{axis}
\end{tikzpicture}
\begin{tikzpicture}[scale=1.8]
\draw[step=0.5cm,color=gray] (-1,-1) grid (1,1);
\node at (-0.75,+0.75) {-1};
\node at (-0.25,+0.75) {1};
\node at (+0.25,+0.75) {-1};
\node at (+0.75,+0.75) {1};
\node at (-0.75,+0.25) {1};
\node at (-0.25,+0.25) {-1};
\node at (+0.25,+0.25) {1};
\node at (+0.75,+0.25) {-1};
\node at (-0.75,-0.25) {-1};
\node at (-0.25,-0.25) {1};
\node at (+0.25,-0.25) {-1};
\node at (+0.75,-0.25) {1};
\node at (-0.75,-0.75) {1};
\node at (-0.25,-0.75) {-1};
\node at (+0.25,-0.75) {1};
\node at (+0.75,-0.75) {-1};
\coordinate[label=below:$t$] (A) at (1.2,-1);
\draw[->](-1,-1) to (A);
\coordinate[label=left:$x$] (B) at (-1,1.2);
\draw[->](-1,-1) to (B);
\coordinate[label=below:$.75$] (A2) at (.5,-1);
\coordinate[label=below:$.5$] (A3) at (0,-1);
\coordinate[label=below:$.25$] (A4) at (-.5,-1);
\coordinate[label=left:$.75$] (A2a) at (-1,.5);
\coordinate[label=left:$.5$] (A3a) at (-1,0);
\coordinate[label=left:$.25$] (A4a) at (-1,-.5);
\coordinate[label=left: ] (A4d) at (-1.4,-1.5);
\coordinate[label=below:$f$ ] (A4ds) at (0,-1.3);
\end{tikzpicture}
\caption{Convergence history plot of the squared residual $\eta^2(\mathcal{T})$ with equally and parabolically scaled as well as uniform (solid line) and adaptive (dotted line) mesh refinements for Experiment 2 in Section \ref{sec:Exp3new} (Checkerboard) and the right-hand side $f:Q\to \lbrace -1, 1\rbrace$.}\label{fig:Exp3new}}
\end{figure}
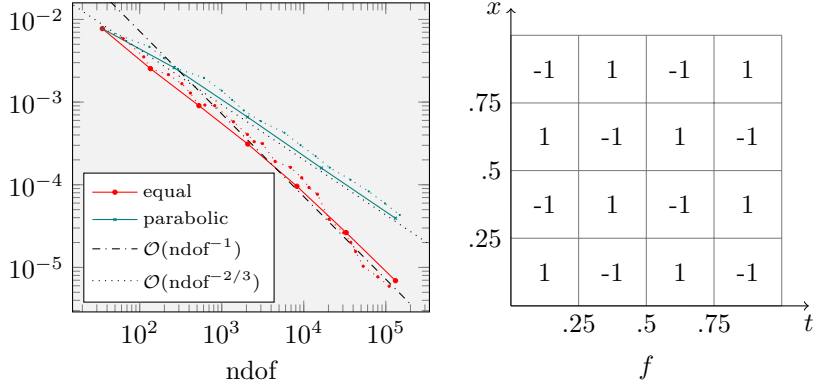
\subsection{Experiment 3 (Irregular Initial Data)}\label{sec:Exp4new}
In this experiment the initial data and the right-hand side read for all $(t,x)\in Q= (0,1)^2$ 
\begin{align*}
u_0(x) = \begin{cases}
-1&\text{for }x<1/2,\\
1&\text{else }
\end{cases}
\qquad\text{and}\qquad
f(t,x) = 0.
\end{align*}
To provide some educated guess for the regularity of $u$, we consider the problem on the entire space $\mathbb{R}$, that is, we solve $\partial_t \hat{u} - \partial_x^2 \hat{u} = 0$ in $\mathbb{R}_{>0}\times \mathbb{R}$ with initial data $\hat{u}(0,\bigcdot) = \sgn$. 
By convolution of the heat kernel \cite[Sec.\ 2.3.1]{Evans10} we get (with Gauss error function $\erf(x) = 2/\sqrt{\pi} \int_0^x \exp(-y^2)\,\mathrm{d}y$) 
\begin{align*}
&\hat{u}(t,x) = \frac{1}{\sqrt{4\pi t}} \int_\mathbb{R} \exp\left(-\frac{(x-y)^2}{4t} \right)\sgn(y)\,\mathrm{d}y\\
& \quad = \frac{2}{\sqrt{4\pi t}} \int_{0}^x \exp\left(-\frac{(x-y)^2}{4t} \right)\,\mathrm{d}y = \erf\left(\frac{x}{2\sqrt{t}}\right)\qquad\text{for all }(t,x) \in \mathbb{R}_{>0}\times \mathbb{R}.
\end{align*}
Differentiation yields for all $(t,x) \in \mathbb{R}_{>0}\times \mathbb{R}$
\begin{align*}
\partial_t \hat{u}(t,x) = \partial_x^2 \hat{u}(t,x) = -\frac{x\exp(-x^2/(4t))}{2t\sqrt{t\pi}} \quad\text{and}\quad \partial_x \hat{u}(t,x) = \frac{\exp(-x^2/(4t))}{\sqrt{t\pi}}.
\end{align*}
Due to the point singularity in $x=t=0$, the function $\partial_t \hat{u}$ (and so $\partial_t u$) is not in $L^2(Q)$. 
Figure \ref{fig:Exp4new} displays the convergence history plot. 
For uniform mesh refinements it indicates the rate of convergence $\eta^2(\mathcal{T}) = \mathcal{O}(\textup{ndof}^{-1/3})$ for parabolically scaled refinements and $\eta^2(\mathcal{T}) = \mathcal{O}(\textup{ndof}^{-1/4})$ for equally scaled refinements.
Such poor convergence results are observed for similar problems by Andreev \cite[Fig.\ 3]{Andreev13} as well as F\"uhrer and Karkulik \cite[Sec.\ 5.2.3]{FuehrerKarkulik19}. 
Indeed, the observed rate for equally scaled mesh refinements equals the observed rate of convergence in \cite[Sec.\ 5.2.3]{FuehrerKarkulik19} derived with a LSFEM scheme. In contrast to the LSFEM scheme, adaptivity improves the convergence of our space-time DPG method significantly. The convergence history plot suggests the convergence $\eta^2(\mathcal{T}) = \mathcal{O}(\textup{ndof}^{-1})$ for adaptive parabolically scaled refinements and a slightly worse rate for adaptive equally scaled refinements. Notice that the computation with equally scaled mesh refinements experiences a significantly larger pre-asymptotic regime than the computation with parabolically scaled mesh refinements.
\begin{figure}
{\flushright
\begin{minipage}{.5\textwidth}
\includegraphics[width=1.2\textwidth]{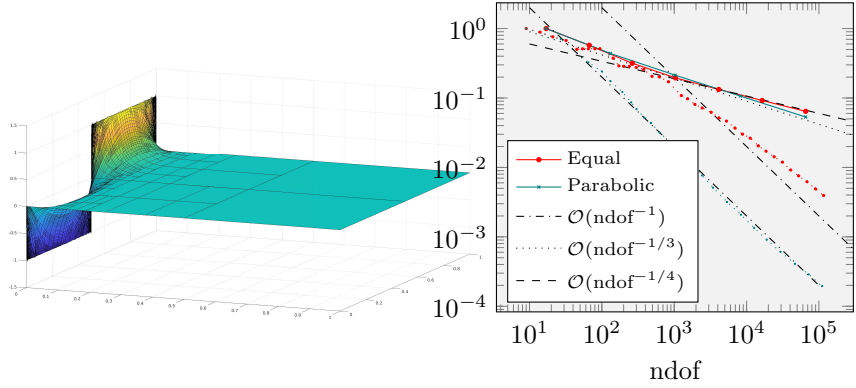}
\end{minipage}
\hspace*{-.3cm}
\begin{minipage}{.5\textwidth}
\begin{tikzpicture}
\begin{axis}[
clip=false,
width=\textwidth,
height=.9\textwidth,
xmode = log,
ymode = log,
xlabel={ndof},
cycle multi list={\nextlist MyColors1bc},
scale = {1},
clip = true,
legend cell align=left,
legend style={legend columns=1,legend pos= south west,font=\fontsize{7}{5}\selectfont}
]
	\addplot table [x=ndof,y=residual_sq] {Experiments/Data_Exp3_uniform.txt};
	\addplot table [x=ndof,y=residual_sq] {Experiments/Data_Exp3_parabolic_uniform.txt};
	\addplot[dash dot,sharp plot,update limits=false] coordinates {(1e2,2e0) (1e6,2e-4)};	
	\addplot[dotted,sharp plot,update limits=false] coordinates {(1e1,9e-1) (1e7,9e-3)};  
	\addplot[dashed,sharp plot,update limits=false] coordinates {(1e1,6e-1) (1e9,6e-3)};  
	\addplot table [x=ndof,y=residual_sq] {Experiments/Data_Exp3_adaptive.txt};
	\addplot table [x=ndof,y=residual_sq] {Experiments/Data_Exp3_parabolic_adaptive.txt};	
		\addplot[dash dot,sharp plot,update limits=false] coordinates {(1e1,2e0) (1e5,2e-4)};
	\legend{
	{Equal},
	{Parabolic},
	{$\mathcal{O}(\textup{ndof}^{-1}$)},
	{$\mathcal{O}(\textup{ndof}^{-1/3}$)},
	{$\mathcal{O}(\textup{ndof}^{-1/4}$)}};
\end{axis}
\end{tikzpicture}
\

\end{minipage}}
\caption{Solution $\hat{u}_h \in V_h$ with underlying partition $\mathcal{T}$ and $\textup{ndof} = 5385$ as well as the convergence history of the residual $\eta^2(\mathcal{T})$ with uniform mesh refinements (solid line) and adaptive mesh refinements (dotted lines) with equally and parabolically scaled mesh refinements for Experiment 3 (Irregular Initial Data).}\label{fig:Exp4new}
\end{figure}
\subsection{Experiment 4 (Rough Right-Hand-Side)}\label{sec:Exp4}
In this experiment the initial data equals $u_0 = 0$. We investigate our numerical scheme with two singular right-hand sides. They read for all $(t,x)\in Q$ and parameters $\alpha \leq 0$ 
\begin{align*}
\begin{aligned}
\bff &= (f_x^\alpha,0)&&\quad\text{with }f_x^\alpha(t,x) = |x-1/2|^\alpha, \\
\bff & = (f_t^\alpha,0)&&\quad\text{with }f_t^\alpha(t,x) = |t-1/2|^\alpha.
\end{aligned}
\end{align*}
Since the right-hand side satisfies $f\in L^{|1/\alpha| - \varepsilon}(Q)$, we expect the maximal regularity $\partial_t u,\Delta_x u \in L^{|1/\alpha| - \varepsilon}(Q)$ for all $\alpha \leq 0$ and $\varepsilon>0$ \cite{Fackler17}. 
Therefore, the estimate in \eqref{eq:term1} suggest the failure for equally refined meshes as $\alpha \searrow 1/2$.
Table \ref{tab:RatesOfConv} displays the rates of convergence of the  squared residual $\lVert (\eta_h,\vartheta_h) \rVert_Y^2$ for various values of $\alpha$ and the right-hand sides $f_x^\alpha$ (``Singular in Space'') and $f_t^\alpha$ (``Singular in Time''). We compute the rates via linear regression of the log log convergence history plot, where the number of degrees of freedom on the initial partition $\mathcal{T}$ equals $\textup{ndof} = 9$ and
is larger than $10^5< \textup{ndof}$ on the finest partition.
We observe that the singularity causes severe difficulties for equally scaled mesh refinements in the sense that the rate of convergence goes for uniform and adaptive mesh refinements to zero as $\alpha \searrow-1/2$ (the limiting cases for $\bff \not\in L^2(Q;\mathbb{R}^{d+1})$). The parabolically scaled scheme works significantly better: the rate of convergence for uniform and adaptive refinements is above $2/3$ for $\alpha \geq -1/2$, which is the optimal rate for smooth solutions. 
This underlines the theoretical considerations in Section \ref{sec:EquiVsPara}.
Notice that we have to add the data approximation error (not displayed in this paper) to the residuals in Table \ref{tab:RatesOfConv} to conclude decay rates for the error. For the right-hand side $f_x^\alpha$ and $\alpha\geq -0.4$ the data approximation error for the parabolically scaled uniform and adaptive refinements converges with a rate faster than $2/3$, for $\alpha = -0.5$ the rate equals $0.55$. For the right-hand side $f_t^\alpha$ the convergence rate of the data approximation error goes to zeros as $\alpha \searrow -1/2$.
Nevertheless, this experiment indicates that parabolically scaled mesh refinements are obligatory for rough solutions.

\begin{table}
{\centering
\begin{tabular}{c|c|c|c|c|c|c|c|c|c}
\multicolumn{1}{c}{}&\multicolumn{4}{c|}{Singular in Space}& & \multicolumn{4}{c}{Singular in Time}\\
 \multicolumn{1}{c}{}&\multicolumn{2}{c|}{Equal}& \multicolumn{2}{c|}{Parabolic}& &\multicolumn{2}{c|}{Equal}& \multicolumn{2}{c}{Parabolic}\\ \hline
$\alpha$& unif& adapt&  unif& adapt&&unif& adapt&  unif& adapt \\ \hline
0.00&1&	1.01&	0.66&	0.66&& 1&	1.01&	0.66&	0.66\\
-0.05	&	0.81	&	0.93	&	0.67	&	0.67	&	&	0.81	&	0.93	&	0.66	&	0.66	\\
-0.1	&	0.6	&	0.83	&	0.67	&	0.68	&	&	0.59	&	0.82	&	0.67	&	0.65	\\
-0.15	&	0.46	&	0.71	&	0.67	&	0.7	&	&	0.45	&	0.7	&	0.67	&	0.67	\\
-0.2	&	0.36	&	0.6	&	0.68	&	0.71	&	&	0.36	&	0.59	&	0.67	&	0.68	\\
-0.25	&	0.29	&	0.49	&	0.69	&	0.72	&	&	0.28	&	0.48	&	0.68	&	0.69	\\
-0.3	&	0.22	&	0.39	&	0.69	&	0.75	&	&	0.22	&	0.38	&	0.69	&	0.72	\\
-0.35	&	0.17	&	0.29	&	0.69	&	0.76	&	&	0.16	&	0.28	&	0.7	&	0.74	\\
-0.4	&	0.11	&	0.19	&	0.67	&	0.76	&	&	0.11	&	0.18	&	0.7	&	0.77	\\
-0.45	&	0.06	&	0.08	&	0.65	&	0.76	&	&	0.06	&	0.07	&	0.69	&	0.79	\\
-0.5	&	0	&	-0.02	&	0.62	&	0.76	&	&	0	&	-0.03	&	0.65	&	0.81	\\
-0.55	&	-0.05	&	-0.12	&	0.59	&	0.73	&	&	-0.05	&	-0.12	&	0.6	&	0.79	\\
-0.6	&	-0.1	&	-0.21	&	0.55	&	0.71	&	&	-0.1	&	-0.22	&	0.53	&	0.78	\\
-0.65	&	-0.15	&	-0.31	&	0.51	&	0.68	&	&	-0.15	&	-0.32	&	0.47	&	0.64	\\
-0.7	&	-0.2	&	-0.41	&	0.47	&	0.62	&	&	-0.2	&	-0.4	&	0.4	&	0.58	\\
-0.75	&	-0.25	&	-0.5	&	0.43	&	0.56	&	&	-0.25	&	-0.48	&	0.33	&	0.42
\end{tabular}
\caption{Convergence rates of the squared residual $\lVert (\eta_h,\vartheta_h) \rVert_Y^2$ in Experiment 4 (Rough RHS) for various parameters $\alpha$, equally and parabolicly scaled meshes as well as uniform (unif) and adaptive (adapt) mesh refinements.}\label{tab:RatesOfConv}}
\end{table}
\bibliographystyle{amsalpha}
\bibliography{SpaceTimeDPG}

\end{document}